\documentclass[10pt]{amsart}
\usepackage{amssymb}
\usepackage{amsmath}
\usepackage{amsfonts}

\newtheorem{theorem}{Theorem}[section]
\newtheorem{definition}[theorem]{Def\/inition}
\newtheorem{proposition}[theorem]{Proposition}
\newtheorem{corollary}[theorem]{Corollary}
\newtheorem{lemma}[theorem]{Lemma}
\newtheorem{rmk}[theorem]{Remark}
\newenvironment{remark}{\begin{rmk}\rm}{\end{rmk}}
\newtheorem{conjecture}[theorem]{Conjecture}

\newcommand{\iso}{\smash{\mathop{\longrightarrow}\limits^{\sim}}}

\newcommand{\ol}{\overline}

\newcommand{\A}{\mathbf{A}}
\newcommand{\C}{\mathbf{C}}
\newcommand{\F}{\mathbf{F}}
\newcommand{\Fl}{\F_\ell}
\newcommand{\Flbar}{\overline{\F}_\ell}
\newcommand{\Kbar}{\overline{K}}
\newcommand{\N}{\mathbf{N}}
\newcommand{\Q}{\mathbf{Q}}
\newcommand{\Qbar}{\overline{\Q}}
\newcommand{\Qlbar}{\overline{\Q}_\ell}
\newcommand{\R}{\mathbf{R}}
\newcommand{\Z}{\mathbf{Z}}
\renewcommand{\P}{\mathbf{P}}
\newcommand{\Gm}{\mathbf{G}_m}

\DeclareMathOperator{\Res}{Res}
\DeclareMathOperator{\End}{End}
\DeclareMathOperator{\Frob}{Frob}
\DeclareMathOperator{\GL}{GL}
\DeclareMathOperator{\Gal}{Gal}
\DeclareMathOperator{\Hom}{Hom}
\DeclareMathOperator{\M}{M}
\DeclareMathOperator{\new}{new}
\DeclareMathOperator{\Pic}{Pic}
\DeclareMathOperator{\sgn}{sgn}
\DeclareMathOperator{\SL}{SL}
\DeclareMathOperator{\Spec}{Spec}
\DeclareMathOperator{\Symm}{Symm}
\DeclareMathOperator{\T}{\mathbf{T}}

\def\smallmat#1#2#3#4{\bigl(\begin{smallmatrix}{#1}&{#2}\\{#3}&{#4}\end{smallmatrix}\bigr)}

\newcommand{\Ind}{\mathrm{Ind}}
\newcommand{\gr}{\mathrm{gr}}
\newcommand{\fil}{\mathrm{Fil}}
\newcommand{\ur}{\mathrm{ur}}
\newcommand{\tor}{\mathrm{tor}}
\newcommand{\crys}{\mathrm{crys}}
\newcommand{\et}{\mathrm{\acute{e}t}}
\newcommand{\ab}{\mathrm{ab}}

\newcommand{\Ext}{\mathrm{Ext}}
\newcommand{\val}{\mathrm{val}}
\newcommand{\tr}{\mathrm{tr}}
\newcommand{\spec}{\Spec}
\newcommand{\red}{\mathrm{red}}

\newcommand{\CO}{\mathcal{O}}
\newcommand{\CF}{\mathcal{F}}

\newcommand{\gn}{\mathfrak{n}}
\newcommand{\gm}{\mathfrak{m}}
\newcommand{\gp}{\mathfrak{p}}
\newcommand{\gq}{\mathfrak{q}}
\newcommand{\frob}{\mathsf{Frob}}
\newcommand{\uhp}{\mathfrak{H}}

\begin{document}

\title[On Serre's conjecture over totally real fields]{On Serre's conjecture for mod $\ell$ Galois representations over totally real fields}

\author{Kevin Buzzard}
\address{Department of Mathematics, Imperial College, 180 Queen's Gate,
London SW7 2AZ, UK}
\email{k.buzzard@imperial.ac.uk}
\author{\sc Fred Diamond}
\address{Department of Mathematics,
King's College London, Strand,
London WC2R 2LS, UK}
\email{fred.diamond@kcl.ac.uk}
\author{\sc Frazer Jarvis}
\address{Department of Pure Mathematics, University of Sheffield,
Sheffield S3 7RH, UK}
\email{A.F.Jarvis@sheffield.ac.uk}

\date{April 2010}
\maketitle
\begin{abstract}
In 1987 Serre conjectured that any mod $\ell$ two-dimensional irreducible odd representation of the absolute Galois group of the rationals came from a modular form in a precise way. We present a generalisation of this conjecture to 2-dimensional representations of the absolute Galois group of a totally real field where $\ell$ is unramified. The hard work is in formulating an analogue of the ``weight'' part of Serre's conjecture. Serre furthermore asked whether his conjecture could be rephrased in terms of a ``mod $\ell$ Langlands philosophy''. Using ideas of Emerton and Vign\'eras, we formulate a mod $\ell$ local-global principle for the group $D^*$, where $D$ is a quaternion algebra over a totally real field, split above $\ell$ and at 0 or 1 infinite places, and show how it implies the conjecture.
\end{abstract}

\section{Introduction}
Serre conjectured in \cite{serre:duke} that if $\ell$ is prime and 
$$\rho: G_\Q \to \GL_2(\ol{\F}_\ell)$$
is a continuous, odd, irreducible representation, then $\rho$
is {\em modular} in the sense that it arises as the reduction of an $\ell$-adic
representation associated to a Hecke eigenform in the space $S_k(\Gamma_1(N))$
of cusp forms of some weight $k$ and level $N$. Let us refer to this
(incredibly strong) conjecture as ``the weak conjecture''.
Serre goes on to formulate a refined conjecture which predicts the minimal
weight and level of such an eigenform subject to the constraints $k \ge 2$
and $\ell \nmid N$; let us call this ``Serre's refined conjecture''.
Note that Serre explicitly excludes weight~1 modular
forms, although a further reformulation was made by Edixhoven
in~\cite{edixhoven} to include them, and we refer to Edixhoven's
reformulation as ``Edixhoven's refined conjecture''.
Through the work of Ribet~\cite{ribet:inv}, Gross~\cite{gross},
Coleman-Voloch~\cite{cv} and others, the equivalence
between the weak conjecture and Serre's refinement was known for $\ell>2$
(see~\cite{fred:hk}), and also when $\ell=2$ in many cases 
(see~\cite{buzzard:2}). The equivalence of Serre's refined conjecture
and Edixhoven's refined conjecture is also essentially known,
although the question does not appear to have been completely
resolved: for $\ell = 2$, there still appears to be an issue regarding constructing
a weight~1 form in every case that Edixhoven predicts that such a form exists.

The aim of this paper is to formulate a generalisation of Serre's
refined conjecture to the context of two-dimensional representations
of $G_K$ where $K$ is a totally real field. The details of such
a formulation (assuming $\ell$ unramified in $K$) were worked out by one
of us (F.D.) stemming from correspondence
and conversations among the authors in 2002, and the first version
of this paper appeared in 2004. At that time the weak conjecture (for $G_\Q$)
appeared out of reach. Since then there has been startling progress,
culminating in its recent proof by Khare and Wintenberger~\cite{kw:serre1,
kw:serre2}, building on ideas developed by Dieulefait and 
themselves~\cite{dieulefait, wintenberger, kw:serre0, khare:serre} and relying
crucially on potential modularity and modularity lifting methods and results of Taylor,
Wiles and Kisin~\cite{wiles:flt, tw:flt, rlt:pm, kisin:annals, kisin:inv}.
Their result also resolves the remaining case for $\ell=2$ of Serre's refined
conjecture.

Since the first version of this paper appeared, there has also been 
significant progress towards proving the equivalence between the
``weak'' and ``refined'' conjectures we presented over $G_K$. 
Partial results already followed from work of one of the authors~\cite{frazer:mazur, frazer:ll},
Fujiwara~\cite{fujiwara} and Rajaei~\cite{rajaei}, and further results were subsequently obtained
by Schein~\cite{schein:crelle} and Gee~\cite{gee:thesis}.
For the most part the techniques were generalisations of ones already used
in the case $K=\Q$ and seemed severely limited with respect to
establishing the ``weight part'' of the refined conjecture.
However in~\cite{gee:duke, gee:ppt} Gee presented a new, much more promising approach;
it remains to be seen how far the ideas there
can be pushed towards a complete proof of the equivalence between
weak and refined conjectures. Note also that these ideas have
been extended in~\cite{gee-savitt} to cover cases where $\ell$ is
ramified in~$K$.

Another important development related to Serre's conjecture
has been the recent progress on constructing $\ell$-adic
and mod $\ell$ Langlands correspondences, especially
the work of Breuil, Colmez and Emerton. In particular,
a correspondence between two-dimensional $\ell$-adic
representations of $G_{\Q_\ell}$ and certain $\ell$-adic
representations of $\GL_2(\Q_\ell)$ was constructed by Colmez~\cite{colmez},
and a conjectural compatibility with a global correspondence 
was formulated and proved in many cases by Emerton in~\cite{emerton:draft-a}
(see also~\cite{emerton:coates}). 
There is also a mod $\ell$ version of this compatibility, 
which we refer to as ``Emerton's refined conjecture''
(see also \cite{emerton:draft-b}). Most cases of Serre's refined conjecture
follow from Emerton's; in particular the specification of the weight
is essentially a description of the $\GL_2(\Z_\ell)$-socle of the 
representation of $\GL_2(\Q_\ell)$ arising as a local factor at $\ell$
associated to $\rho$.
The current version of this paper includes a partial generalisation
of Emerton's refined conjecture to certain forms of $\GL_2$ over $K$. 
The sense in which it is ``partial'' is that at primes over $\ell$ 
we only describe the Jordan-H\"older constituents of the
socle of a maximal compact for the local factor,
and even that only for primes unramified over $\ell$.
The relevance of this socle and the difficulty
of generalising the mod $\ell$ local correspondence to
extensions of $\Q_\ell$ is clearly illustrated by the recent work of
Breuil and Paskunas~\cite{breuil_paskunas}.

We now explain our set-up and aims in a little more detail. 
Suppose that $K$ is a totally real field. Let $\CO$ denote its ring
of integers and let $S_K$ be the set of embeddings of $K$ in $\R$.
Suppose that $\vec{k} \in \Z^{S_K}$ with $k_\tau \ge 1$ for all $\tau\in S_K$
and furthermore assume that all of the $k_\tau$ are of the same
parity. Let $\gn$ be a non-zero ideal of $\CO$. The space of Hilbert
modular cusp forms of weight $\vec{k}$ and level $\gn$, denoted
$S_{\vec{k}}(U_1(\gn))$, is a finite-dimensional complex vector space
equipped with an action of commuting Hecke operators $T_\gm$, indexed by the
non-zero ideals $\gm$ of $\CO$ (to fix ideas, let us normalise our spaces
and Hecke operators as in~\cite{rlt:inv}). Fix once and for all
embeddings $\ol{\Q}\to\C$ and
$\ol{\Q}\to\ol{\Q}_\ell$, and let $0\not=f\in S_{\vec{k}}(U_1(\gn))$
be an eigenform for all the $T_\gm$.
A construction of Rogawski-Tunnell, Ohta and Carayol~\cite{rt:inv, ohta, carayol:ens}, 
completed by Taylor and Jarvis~\cite{rlt:inv, jarvis:wt1},
associates to $f$ an $\ell$-adic representation
$$\rho_f: G_K \to \GL_2(\ol{\Q}_\ell)$$
such that if $\gp$ is a prime of $\CO$ not dividing
$\ell \gn$, then $\rho_f$
is unramified at $\gp$ and, if $\frob_\gp$ denotes a geometric Frobenius,
then $\mathrm{tr}\rho_f(\frob_\gp)$ is the eigenvalue
of $T_\gp$ on~$f$ (note that Taylor does not need to
specify whether his Frobenius elements are arithmetic or geometric, so we shall
assume that they are geometric).
Fixing an identification of the residue field of
$\ol{\Q}_\ell$ with $\ol{\F}_\ell$,
we obtain a representation
$$\ol{\rho}_f: G_K \to \GL_2(\ol{\F}_\ell)$$
defined as the semisimplification of the reduction of $\rho_f$.
It is natural to expect the following ``folklore'' generalisation of 
Serre's weak conjecture to hold:
\begin{conjecture}\label{conj:weak}
Suppose that $\rho:G_K \to \GL_2(\ol{\F}_\ell)$ is
continuous, irreducible and totally odd. Then $\rho$
is isomorphic to $\ol{\rho}_f$
for some Hilbert modular eigenform $f$.
\end{conjecture}

Here ``totally odd'' means that $\det(\rho(c))=-1$ for all $[K:\Q]$
complex conjugations~$c$. Note that one could instead have defined $\rho_f$ to
be the representation with the property that the trace of an \emph{arithmetic}
Frobenius was equal to the corresponding Hecke eigenvalue, which
is the same as replacing $\rho_f$ by its dual,
but the ``geometric'' conjecture above is trivially
equivalent to the ``arithmetic'' version ($\rho$ is geometrically
modular if and only if its dual is arithmetically modular). 
Although the Khare-Wintenberger approach to Serre's original conjecture may
shed light on Conjecture~\ref{conj:weak} for a few explicit totally real
fields~$K$, it seems not
(in its present form) to be able to attack the case of a general~$K$ because it
relies on an induction and the fact that for certain small primes $\ell$
there are no 2-dimensional irreducible odd mod $\ell$ representations
of $G_{\Q}$ unramified outside $\ell$; however for a totally real field
the analogous fact is not in general true---there are plenty of elliptic
curves over totally real fields with everywhere good reduction, and the
$\ell$-torsion in these curves will generally give rise to such representations.
As an explicit example, one can check easily using a computer that if
$\epsilon=\frac{5+\sqrt{29}}{2}$ then 
the elliptic curve
$$y^2+xy+\epsilon^2y=x^3$$
has everywhere good reduction over $\Q(\sqrt{29})$ and that the Galois
representation on the 2-torsion is absolutely irreducible.
\footnote{Prof. J.-P.\ Serre informs us that this example of a curve
with everywhere good reduction was found by Tate, and notes that it
is discussed in \S5.10 of Serre's 1972 Inventiones
article ``Propri\'et\'es galoisiennes des points d'ordre fini des courbes
elliptiques'' (Oeuvres 94).}

The main aim of this paper is to refine Conjecture~\ref{conj:weak}
along the lines of Serre's refinement for the case $K = \Q$,
in the special case where the prime $\ell$ is unramified in~$K$.
Perhaps surprisingly, this is not as simple as it sounds.
The main difficulty is in specifying the weight where, even in this
unramified situation, we encounter
several subtleties not present in Serre's original work. Note first of all
that there is no obvious notion of a minimal weight. Moreover
the possible weights and level structures at primes
over~$\ell$ are intertwined, and, contrary to the case $K=\Q$,
one does \emph{not} always expect a
representation as in Conjecture~\ref{conj:weak} to arise from a 
classical Hilbert modular form of level prime to $\ell$. Indeed,
the mod~$\ell$ representation attached to such a form
has determinant equal to the product of a finite order character unramified
at~$\ell$ and some power of the mod~$\ell$ cyclotomic
character, and it is not hard to construct a mod $\ell$ Galois representation
whose determinant is not of this form. To deal with these issues,
we introduce the notion of a {\em Serre weight},
namely an irreducible $\ol{\F}_\ell$-representation~$\sigma$
of $\GL_2(\CO/\ell)$,
and define what it means for $\rho$ to be modular of weight $\sigma$.
Such a notion of weight is implicit in work of Ash and Stevens
\cite{as1, as2}, its relation to Serre's conjecture underlies 
Khare's paper~\cite{khare:pp}, and its role in generalizing the
conjecture to $\GL_n$ over $\Q$ is evident in \cite{ash2, ash1}.
Our aim is to describe all possible Serre weights
of forms giving rise to a representation $\rho$.

When working with classical modular forms, certain choices for normalisations
and conventions have now become standard. In the Hilbert case there
are various possibilities for these choices, and experience has
shown the authors that things become ultimately less
confusing if one works with holomorphic automorphic representations
as in~\cite{carayol:ens} rather than Hilbert modular forms, the advantage
of this approach being that now the \emph{only} choice one has to make
is the normalisation of the local Langlands correspondence. We follow
Carayol in our approach and use Hecke's normalisation rather than
Langlands'. We explain our conventions more carefully later on.

Our weight conjecture (Conjecture~\ref{conj:weight} below) then takes the 
form of a recipe for the set of weights $\sigma$ for which $\rho$ is modular. 
Our notion of modularity is formulated in terms of $\rho$ arising in the
Jacobian (or equivalently cohomology) of Shimura curves associated to
quaternion algebras over $K$;
the weight recipe is given in terms of the local behaviour of $\rho$ at primes over
$\ell$.  The precise recipe is motivated by the expected compatibility with
the (labelled) Hodge-Tate weights in $[0,\ell]$ of crystalline lifts of
twists of $\rho|_{G_{K_\gp}}$ for primes $\gp|\ell$.  One can use
Fontaine-Laffaille theory \cite{fon_laf} to describe the reductions of
crystalline representations with Hodge-Tate weights in the range $[0,\ell-1]$;
that the resulting recipe extends to the ``boundary'' is supported by
numerical evidence collected by
one of the authors, Demb\'el\'e and Roberts~\cite{ddr}.

Loosely speaking, Conjecture~\ref{conj:weak} can be thought of in the
context of a ``mod $\ell$ Langlands philosophy'', with Conjecture~\ref{conj:weight} 
predicting a local-global compatibility at primes over $\ell$.
An interesting feature of the recipe for the weights is that if $\ell$ is inert
in $K$ and $\rho|_{G_{K_\ell}}$ is semisimple, then the set of weights that
we associate to $\rho$ are the Jordan--H\"older factors of the reduction
of an \emph{irreducible} $\ol{\Q}_\ell$-representation
of $\GL_2(\CO/\ell)$.
This is proved in \cite{fred:durham}, where it is further shown that this association
establishes a correspondence between $2$-dimensional
Galois representations of a local field {\em in its residue characteristic}
and representations of $\GL_2$ of its residue field
{\em in characteristic zero}. Herzig~\cite{herzig}
has shown that this phenomenon does not
persist in the context of $\GL_n$ for $n > 2$, but rather
is a property particular to $n\le 2$ of a more general 
relation that he conjectures
between the set of Serre weights and the reduction of a characteristic
zero representation associated to $\rho$.

In \cite{emerton:draft-a, emerton:draft-b}, Emerton made precise
the sense in which Serre's refined conjecture
could be viewed as part of a mod $\ell$ Langlands philosophy (in the case
$K=\Q$). Using automorphic forms, he associates to $\rho$ an $\Flbar$-representation $\pi(\rho)$
of $\GL_2(\A_\Q)$, which is non-zero by the theorem of
Khare and Wintenberger. Emerton conjectures, and shows under some technical hypotheses,
that it factors as a restricted tensor product of local factors $\pi_p$, where $\pi_p$ is
a smooth admissible representation of $\GL_2(\Q_p)$ determined by $\rho|_{G_{\Q_p}}$.
Serre's refined conjecture can then be recovered from properties of the $\pi_p$;
moreover, results such as those in \cite{dt:duke} and \cite{khare:pp} describing the possible
weights and levels of forms giving rise to $\rho$ can also be extracted.
We go on to formulate a conjecture in the spirit of Emerton's in the context of
certain quaternion algebras over $K$. In order to do so, we need to associate
a local factor to $\rho|_{G_{K_\gp}}$ when $\gp$ is a prime not dividing $\ell$.
This was already done by Emerton if the quaternion algebra is split at $\ell$;
we augment this with a treatment of the case where it is ramified using results
of Vign\'eras~\cite{vigneras:quat}.
 
This paper is structured as follows. In \S\ref{sec:setup} we introduce the
notion of a Serre weight and our notation and conventions regarding
automorphic representations for $\GL_2$ over $K$;
we explain what it means for $\rho$ to be modular of
a given Serre weight
and relate this notion to the existence of automorphic representations $\pi$
such that $\rho \sim \ol{\rho}_\pi$.
In \S\ref{sec:recipe} we formulate Conjecture~\ref{conj:weight} giving a
recipe for the set of Serre weights for which $\rho$ is modular,
and we show that it follows from known results on Serre's Conjecture
in the case $K=\Q$ (Theorem \ref{thm:Qconj}). 
Finally, in \S\ref{sec:refined} we state our partial generalisation
of Emerton's refined conjecture and derive some consequences.

\noindent{\bf Acknowledgements:} 
Much of the research was carried out while one of the authors (F.D.) was at Brandeis
University, with support from NSF grants~\#9996345, 0300434.
K.B.\ was partially supported by an EPSRC Advanced Research Fellowship.
We are grateful to R.~Taylor for raising our interest in
the questions considered here. The last section of the paper was greatly influenced
by discussions at a workshop at the American Institute of Mathematics on $p$-adic
Representations in February~2006 attended by K.B.\ and F.D.; we heartily thank
AIM and the workshop organisers and participants, especially M.~Emerton.
The influence of his ideas on that part of this work is apparent; we are also
grateful to him for subsequent correspondence, in particular communicating
details of \cite{emerton:draft-a, emerton:draft-b}.
We also benefited from conversations and correspondence with C.~Breuil, B.~Conrad, 
L.~Demb\'el\'e and T.~Gee in the course of writing this paper.
Finally the authors thank F.~Herzig, M.~Schein and R.~Torrey for their
helpful comments and corrections on an earlier version of the paper,
and of course the referees, who between them pointed out numerous typos
and made several remarks which helped improve the presentation.
Any remaining errors are of course the fault of the authors.

\section{Serre weights}\label{sec:setup}
Suppose $K$ is a totally real field (we allow $K=\Q$)
and let $\CO$ denote its
ring of integers. Let $\ell$ be a rational prime, which we assume from the
outset is unramified in $K$ (although some of this section certainly
could be made to work in more generality). Recall that we have
fixed embeddings $\Qbar\to\C$ and $\Qbar\to\Qlbar$, and also an
identification of $\Flbar$ with the residue field of $\Qlbar$. Let
$S_K$ denote the embeddings $K\to\R$ and let us fix once and for
all a preferred embedding $\tau_0:K\to\R$.

Consider the group
$$G = \GL_2(\CO/\ell\CO) \cong \prod_{\gp|\ell} \GL_2(\CO/\gp).$$
A {\em Serre weight} is an isomorphism class
of irreducible $\ol{\F}_\ell$-representations of $G$. These can be described
explicitly as follows. For each prime $\gp$ of $K$ dividing $\ell$,
set $k_\gp = \CO/\gp$,
$f_\gp = [k_\gp:\F_\ell]$ and let $S_\gp$ be the set of embeddings
$\tau:k_\gp \to \ol{\F}_\ell$. Then every irreducible $\ol{\F}_\ell$-representation
of $\GL_2(k_\gp)$ is equivalent to one of the form
$$V_{\vec{a},\vec{b}} = 
\bigotimes_{\tau\in S_\gp}(\det{}^{a_\tau}\otimes_{k_\gp}\mathrm{Symm}^{b_\tau-1}k_\gp^2)\otimes_\tau\ol{\F}_\ell,$$
where $a_\tau$, $b_\tau \in \Z$ and $1\le b_\tau\le \ell$ for each $\tau \in S_\gp$.
Moreover we can assume that $0 \le a_\tau \le \ell-1$ for each
$\tau \in S_\gp$ and that $a_\tau < \ell - 1$ for some $\tau$,
in which case the resulting $(\ell^{f_\gp}-1)\ell^{f_\gp}$ representations
$V_{\vec{a},\vec{b}}$ are also inequivalent.
The irreducible representations of $G$ are thus of the form
$V = \otimes_{\{\gp|\ell\}} V_\gp$, where the tensor product is over
$\ol{\F}_\ell$ and each $V_\gp$ is of the form $V_{\vec{a},\vec{b}}$
for $(\vec{a},\vec{b})$ as above.

If $n$ is an integer then
we let $\Fl(n)$ denote the 1-dimensional
$\Fl$-vector space with left $G$-action defined by letting $g\in G$
act via $N(\det(g))^n$, where $N:(\CO/\ell\CO)^\times\to(\Z/\ell\Z)^\times$
is the norm. If $\F$ is a field of characteristic $\ell$
and $V$ is an $\F$-representation of $G$
space then we define $V(n)$ to be the $\F$-representation
$V(n):=V\otimes_{\Fl}\Fl(n)$.
Note that $V_{\vec{a},\vec{b}}(n)=V_{\vec{a}+n.\vec{1},\vec{b}}$, where
$\vec{1}=(1,1,\ldots,1)$.

Suppose that $D$ is a quaternion algebra over $K$ split at $\tau_0$ and
at no other infinite places.
Fix an isomorphism $D\otimes_{K,\tau_0}\R=M_2(\R)$; this
induces an isomorphism of $(D\otimes_{\tau_0}\R)^\times$ with $\GL_2(\R)$,
which acts on $\uhp^\pm:=\C\backslash\R$ in the usual way.
Consider $K$ as a subfield
of $\R$ (and hence of $\C$) via the embedding~$\tau_0$.
If $\A_K^f$ denotes the finite adeles of $K$
and $U$ is an open compact subgroup of $(D\otimes_K\A_K^f)^\times$
then
there is a Shimura curve $Y_U$ over $K$, a smooth algebraic curve
whose complex points (via $\tau_0:K\to\C$) are naturally identified with 
$$D^\times\backslash
\left( (D\otimes_K \A^f_K)^\times\times\uhp^\pm\right)
/ U$$
and such that $Y_U$ is a canonical model for this space, in the
sense of Deligne. These canonical models have the useful property
that if $U'$ is a normal compact open subgroup of $U$,
then the natural right action of $U/U'$ on $Y_{U'}(\C)$ is induced
by an action of $U/U'$ on $Y_{U'}$ (that is, the action is
defined over~$K$).

Unfortunately there is more than one convention for these
canonical models, and the choice that we make genuinely affects our
normalisations. To fix ideas, we shall follow the conventions of Carayol
in~\cite{carayol:comp} and in particular our ``Hodge structure''
$h$ will be that of section~0.1 of~\cite{carayol:comp}.
This corresponds to the choice $\epsilon=-1$ in the notation of
\cite{cornut-vatsal}. See Section~3.3 of~\cite{cornut-vatsal}
for a discussion of the differences between this choice
and the other natural choice---the key one being (Lemma 3.12
of~\cite{cornut-vatsal}) that the choice \emph{does} affect the Galois action
on the adelic component group, by a sign. That this ambiguity exists
is not surprising: for example in the elliptic curve case the modular
curve $Y(\ell)$ parametrising elliptic curves equipped with
generators of their $\ell$-torsion exists (for $\ell>2$)
as a moduli space over $\Q$, and the Weil pairing gives a
natural morphism $Y(\ell)\to\Spec(\Q(\zeta_\ell))$, but the two
ways of normalising the Weil pairing give different morphisms.

If $K = \Q$ and $D$ is split (we refer to this case as ``the split case''),
we let $X_U$ denote the standard compactification
of the modular curve $Y_U$; otherwise (that is, if $K\not=\Q$ or
if $K=\Q$ but $D$ is non-split) we simply set $X_U=Y_U$.
Then $X_U$ is a smooth projective algebraic curve over $K$. Note that
$X_U$, considered as a scheme over $K$, will be connected (see section~1.3
of~\cite{carayol:comp}) but not in general geometrically connected.
Note also that, with notation as above, the natural action of $U/U'$ 
on $Y_{U'}$ extends to an action on $X_{U'}$. Opting to include
the split case does sometimes increase the length of a proof
(we have to verify that ``all errors are Eisenstein'')
but is arguably morally better than presenting
proofs only in the non-split case and then merely asserting that they may be
modified to deal with the split case too.

If $U$ is a compact open subgroup of $(D\otimes_K\A_K^f)^\times$
as above, then let $\Pic^0(X_U)$ denote the identity component
of the relative Picard scheme of $X_U\to\Spec(K)$. This definition
is chosen specifically to deal with the fact
that $X_U$ may not be geometrically connected. In more concrete
terms, if $K_U$ denotes the
ring $\Gamma(X_U,\CO_{X_U})$
of globally-defined functions on $X_U$, then $K_U$ is a number field
and a finite abelian extension of $K$, the curve $X_U$ is geometrically
connected when regarded as a scheme over $\Spec(K_U)$, 
and $\Pic^0(X_U)$ is canonically isomorphic to the
restriction of scalars (from~$K_U$ to~$K$) of the Jacobian of $X_U/K_U$.
In particular, $\Pic^0(X_U)$ is an abelian variety over $K$.

We henceforth assume that $D$ is split at all primes of $K$ above $\ell$,
and we fix an isomorphism
$D\otimes_\Q\Q_\ell\cong M_2(K\otimes\Q_\ell)$.
We can now regard $\GL_2(\CO\otimes\Z_\ell)$ as a subgroup
of $(D\otimes_K\A_K^f)^\times$. If $U$ is a compact open subgroup
of $(D\otimes_K\A_K^f)^\times$ as above, and if $\GL_2(\CO\otimes\Z_\ell)$
is contained in~$U$, then we say
that $U$ has \emph{level prime to $\ell$}. In this case, the
natural map $U\to\GL_2(\CO/\ell\CO)=G$ is a surjection. Let $U'$
denote its kernel. Then $U/U'=G$ acts naturally on the right on $Y_{U'}$
and on $X_{U'}$, and hence naturally on the left on $\Pic^0(X_{U'})$ and
$\Pic^0(X_{U'})[\ell](\Kbar)$. Let us say that $U$ is
\emph{sufficiently small}
if it has level prime to $\ell$ and the map $Y_{U'}\to Y_U$ is \'etale
of degree equal to the order of~$G$. Note that any $U$ of level prime
to $\ell$ contains a compact open subgroup that is sufficiently small---this
follows easily from 1.4.1.1--1.4.1.3 of~\cite{carayol:comp} or Lemma~12.1
of~\cite{frazer:mazur}. The induced
map $X_{U'}\to X_U$ will then be finite of degree equal to the order of~$G$
(but it may not be \'etale in the split case---there will usually
be ramification at the cusps).

\begin{definition}\label{def:modular}
Suppose that $\rho: G_K \to \GL_2(\ol{\F}_\ell)$ is a continuous, irreducible
representation and $V$ is a finite-dimensional $\Flbar$-vector
space with a left action of~$G$. We say
that $\rho$ is 
{\em modular of weight $V$} if there is a quaternion algebra $D$ over $K$
split at the primes above $\ell$, at $\tau_0$ and no other infinite places of
$K$, and a sufficiently small open compact subgroup $U$ of
$(D \otimes_K\A_K^f)^\times$ of level prime to $\ell$, such that
$\rho$ is an $\ol{\F}_\ell G_K$-subquotient of 
$\left(\mathrm{Pic}^0(X_{U'})[\ell](\Kbar)\otimes V\right)^G$, where 
$U'=\ker(U\to G)$, $G$ acts diagonally on the tensor product,
and $G_K$ acts trivially on $V$.
\end{definition}

Note that we allow $V$ to be a reducible $G$-representation (out of
convenience at this point in the exposition: we will see in
Lemma~\ref{kevinscontribution}(b)
that $V$ can ultimately be assumed irreducible with no loss of
generality). 
Note also that we allow subquotients with respect to the Galois action
although using Hecke operators and the Eichler-Shimura relation
on $X_{U'}$, one can show that
if we replace ``$G_K$-subquotient'' by ``$G_K$-submodule'' then
the resulting definition is equivalent.
On the other hand, we really want to demand that $\rho$ is a $G_K$-subquotient
of the $G$-invariants
of $\mathrm{Pic}^0(X_{U'})[\ell](\Kbar)\otimes V$ rather than
an $\Flbar[G_K\times G]$-subquotient on which
$G$ acts trivially.
Our conjecture would not be correct were we to use $G$-subquotients;
a general Galois representation would then be modular of more
weights and we would not recover important subtleties of the refined
conjecture.

Say $V$ is any finite-dimensional $\Flbar$-vector space equipped with a left
$G$-action. The right action of $G$ on $Y_{U'}$
enables us to identify $G$ with a quotient of $\pi_1(Y_U,x)$ for
$x$ any geometric point of $Y_U$.
Now a standard construction (see for example A I.7 of~\cite{freitag-kiehl})
associates
to~$V$ a locally constant \'etale sheaf $\CF_V$ on $Y_U$, with (amongst
other things) the property
that the pullback of $\CF_V$ to $Y_{U'}$ is just the constant sheaf
associated to the vector space~$V$. We abuse notation slightly by
also using $\CF_V$ to refer to the pullback of $\CF_V$ to
$Y_{U,\Kbar}$, the base change of $Y_U$ to $\Kbar$.

Later on we will also need an $\ell$-adic variant of this construction,
for which a good reference is sections~2.1.3 and~2.1.4 of~\cite{carayol:ens}.
Let $U\subset(D\otimes_K\A_K^f)^\times$ be compact and open, and let $Y_U$
denote the associated Shimura curve over $K$. 
If $\vec{k}\in\Z_{\geq2}^{S_K}$ and $w\in\Z$ with $w\equiv k_\tau$
mod~2 for all $\tau$, then we would like as in loc.\ cit.\ to
define an $\ell$-adic
sheaf corresponding to the representation
$\otimes_{\gp|\ell}\otimes_{\tau\in 
S_\gp}\det^{(w-k_\tau+2)/2}\Symm^{k_\tau-2}(\xi_\tau)$ of $D^\times$,
where $\xi_\tau$
is the tautological
2-dimensional representation of $(D\otimes_{F,\tau}\C)^\times$
induced by a fixed isomorphism $D\otimes_{K,\tau}\C=M_2(\C)$.
We do this as follows. If $E\subset\Qbar$ is a number field, Galois over $\Q$,
splitting $D$ and containing all embeddings of $K$ into $\Qbar$,
then the representation above can be realised on an $E$-vector space.
If $\lambda|\ell$ is a prime of $E$, then (possibly after replacing~$U$
by a finite index subgroup to remove any problems with elliptic
points) Carayol defines an $\CO_{E_\lambda}$-sheaf $\CF^0_{\lambda,U}$ 
on $Y_U$ associated to the representation above, and an $E_\lambda$-sheaf
$\CF_{\lambda,U}=\CF^0_{\lambda,U}\otimes\Q_\ell$.
Strictly speaking the sheaf $\CF^0_{\lambda,U}$ depends on a choice of lattice,
which we always take to be the one arising
from tensor products of symmetric powers of $\CO_{E_\lambda}^2$.
Furthermore, if we choose $\lambda$ so that it
is the prime of $E$ above $\ell$ induced by our embedding $E\to\Qbar\to\Qlbar$
then there is an induced map $\CO_{E_\lambda}/\lambda\to\Flbar$
and the induced $\Flbar$-sheaf $\CF^0_{\lambda,U}/\lambda\otimes\Flbar$
is isomorphic to the sheaf $\CF_W$
associated to the representation
$W=\otimes_{\gp|\ell}\otimes_{\tau\in 
S_\gp}\det^{(w-k_\tau+2)/2}\Symm^{k_\tau-2}k_\gp^2\otimes_\tau\Flbar$.

We begin by noting that $G_K$ acts via an abelian quotient on many
of the cohomology groups that show up in forthcoming arguments.

\begin{lemma}\label{eisenstein}
(a) Let $U$ be any compact open subgroup of $(D\otimes_K\A_K^f)^\times$
and let $\CF$ be a locally constant torsion sheaf on $Y_U$ corresponding
to a continuous representation of $U/U'$ for some normal open compact
$U' \subseteq U$ such that $Y_{U'} \to Y_U$ is \'etale with covering
group $U/U'$.
Then for $i\in\{0,2\}$ the action of $G_K$ on the
cohomology groups $H^i(Y_{U,\Kbar},\CF)$ and $H^i_c(Y_{U,\Kbar},\CF)$
factors through an abelian quotient.

(b) If $\CF_{\lambda,U}$ is the sheaf associated to the data $(\vec{k},w)$
as above, and if $H^1_p(Y_{U,\Kbar},\CF_{\lambda,U})$
denotes the image of $H^1_c(Y_{U,\Kbar},\CF_{\lambda,U})$
in $H^1(Y_{U,\Kbar},\CF_{\lambda,U})$,
then the action of $G_K$ on
$H^1(Y_{U,\Kbar},\CF_{\lambda,U})/H^1_p(Y_{U,\Kbar},\CF_{\lambda,U})$
is via an abelian quotient.

(c) The action of $G_K$ on the cokernel of the natural inclusion
$H^1(X_{U,\Kbar},\F_\ell)\to H^1(Y_{U,\Kbar},\F_\ell)$ is via
an abelian quotient.
\end{lemma}
\begin{proof}

(a) The pullback of $\CF$
to $Y_{U'}$ is constant and
$H^0(Y_{U,\Kbar},\CF)$ can be identified with a subspace of
$H^0(Y_{U',\Kbar},\CF)$. The Galois action on this latter space
is abelian, as the geometric components of $Y_{U'}$ are defined
over an abelian extension of $K$ (see for example \S1.2
of~\cite{carayol:comp}). This proves the result for $H^0$,
and the result for $H^0_c$ follows as $H^0_c$ is a subgroup of $H^0$.
For $i=2$ the result follows from the $i=0$ case and Poincar\'e duality,
which pairs $H^0$ with $H^2_c$ and $H^2$ with $H^0_c$.

(b) The quotient is trivial in the non-split case, so we are instantly
reduced to the case $K=\Q$, $D=M_2(\Q)$ and $E_\lambda = \Q_\ell$. In this case the result is
surely well-known, but we sketch the proof for lack of a reference.
The sheaf associated to the data $(k,w)$ is $\Q_\ell$-dual to the sheaf
associated to $(k,-w)$, so by Poincar\'e duality, this is equivalent 
to showing the Galois action factors through an abelian quotient
on the kernel of the map (with $w$ replaced by $2-w$).
 Recall that
$H_c^1(Y_{U,\Qbar},\CF_{\lambda,U}) = H^1(X_{U,\Qbar},j_!\CF_{\lambda,U})$
where $j:Y_U \to X_U$ is the natural inclusion, and that our map factors as
$$H^1(X_{U,\Qbar},j_!\CF_{\lambda,U}) \to 
H^1(X_{U,\Qbar},j_*\CF_{\lambda,U}) \to H^1(Y_{U,\Qbar},\CF_{\lambda,U}),$$
the first map being surjective, the second injective.
Therefore it suffices to prove that the action of $G_\Q$ on
$$H^0(X_{U,\Qbar},j_*\CF_{\lambda,U}/j_!\CF_{\lambda,U})
= H^0(Z_{U,\Qbar},i^*j_*\CF_{\lambda,U})$$
factors through an abelian quotient, where $i:Z_U \to X_U$ is the reduced
closed subscheme defined by the cusps. Shrinking $U$ if necessary, we
can assume there is a universal generalised elliptic curve over $X_U$
(in the sense of \cite{deligne-rapoport}),
and we let $s:E_U\to X_U$ denote its restriction to the open subscheme whose
fibres over closed points are its identity components. Then we find that 
$j_*\CF_{\lambda,U}$ is isomorphic to $\Symm^{k-2}(R^1s_*\Q_\ell)((w+k-2)/2)$,
and since $E_U\times_{X_U}Z_U$ is isomorphic to ${\mathbf G}_{{\mathrm m},Z_U}$, we conclude that
$i^*j_*\CF_{\lambda,U}$ is isomorphic to $\Q_\ell((w-k+2)/2)$. The assertion
now follows from the fact that the cusps are defined over an abelian extension
of $\Q$.

(c) is similar to (b), but simpler.
\end{proof}

Let $D$, $U$, $U'$, $G$
and $V$ be as in Definition~\ref{def:modular}.
If $\psi$ is a continuous character of $G_K$ with values in
$\Flbar^\times$ or $\Qlbar^\times$, then we let $\psi_\A$
denote the character of $\A_K^\times$ corresponding to $\psi$
by class field theory (with uniformizers corresponding to
geometric Frobenius elements).

\begin{lemma}\label{fctk}
Suppose that $\psi:G_K\to\Flbar^\times$ is a continuous character
such that $\psi_\A$ is trivial on $\det(U')$, and let 
$\chi$ denote the restriction of $\psi_\A$ to $(\CO\otimes\Z_\ell)^\times$.
Then $H^1(Y_{U,\Kbar},\CF_{V\otimes\chi\circ\det})$
is isomorphic to $H^1(Y_{U,\Kbar},\CF_V)(\psi)$ as $G_K$-modules.
\end{lemma}

\begin{proof}
The restriction of $\psi$ to $G_{K_U}$ corresponds to $\chi$ via 
the isomorphism 
$$G_{K_U}/G_{K_{U'}} \cong \det(U)/\det(U') \cong (\CO/\ell)^\times.$$
Recall from section~1.1.2 of~\cite{carayol:ens} that we have a commutative
diagram
$$\begin{array}{ccc} Y_{U'} &\to &\spec K_{U'}\\\downarrow&&
\downarrow\\
Y_U&\to&\spec K_U\end{array}$$
such that the action of $G = U/U'$ on $Y_{U'}$ is compatible via $\det$
with that of $\det(U)/\det(U')$ on $\spec K_{U'}$. So if we let
$\CF_\chi$ denote the sheaf on $\spec K_U$ corresponding
to the character $\chi$, we see that $\CF_{\chi\circ\det}$ is
isomorphic to the pull-back of $\CF_\chi$ by the 
map $Y_U \to \spec K_U$. As this map induces a bijection on geometric
components, it follows that
$H^0(Y_{U,\Kbar},\CF_{\chi\circ\det})$ is isomorphic to
$H^0(\spec K_U\times_K\Kbar,\CF_\chi)$ as $G_K$-modules, which
in turn is isomorphic to
$\Ind_{G_{K_U}}^{G_K}\Flbar(\chi)$. Therefore
$\Hom_{\Flbar[G_K]}(\Flbar(\psi), H^0(Y_{U,\Kbar},\CF_{\chi\circ\det}))$
is one-dimensional. Let $\alpha$ be the image in
$H^0(Y_{U,\Kbar},\CF_{\chi\circ\det})$ of a non-trivial element.
Note that the restriction of $\alpha$ to each component of
$Y_{U,\Kbar}$ is non-trivial, so cupping with $\alpha$ defines
an isomorphism
$$H^1(Y_{U,\Kbar},\CF_V)\otimes_{\Flbar}\Flbar\alpha
 \to H^1(Y_{U,\Kbar},\CF_{V\otimes\chi\circ\det}).$$
\end{proof}

\begin{lemma}\label{kevinscontribution}
Let $\rho:G_K\to\GL_2(\Flbar)$
be continuous, irreducible and totally odd.

(a) $\rho$ is a $G_K$-subquotient of
$\left(\mathrm{Pic}^0(X_{U'})[\ell](\Kbar)\otimes V\right)^G$
if and only if $\rho$ is a $G_K$-subquotient
of $H^1(Y_{U,\Kbar},\CF_{V(1)})$.

(b) $\rho$ is modular of weight~$W$ (an arbitrary finite-dimensional
$\Flbar[G]$-module) if and only if $\rho$ is modular of
weight~$V$ for some Jordan-H\"older factor of $W$.

\end{lemma}

\begin{proof}
(a) First recall that $\Pic^0(X_{U'})[\ell](\Kbar)
=H^1(X_{U',\Kbar},\Fl)\otimes\mu_\ell$
as $G_K\times G$-modules.
By Lemma~\ref{eisenstein}(c), the action of $G_K$ on the cokernel
of the natural injection
$$H^1(X_{U',\Kbar},\Fl) \to H^1(Y_{U',\Kbar},\Fl)$$
factors through an abelian quotient. It follows that $\rho$
is modular of weight $V$ if and only if $\rho$ is an $\Flbar[G_K]$-subquotient
of $(H^1(Y_{U',\Kbar},\Fl)\otimes V)^G\otimes\mu_\ell$.

The Hochschild-Serre spectral sequence now
gives us an exact sequence of $G_K$-modules
\begin{align*}H^1(G,H^0(Y_{U',\Kbar},\Fl)\otimes V) \to H^1(Y_{U,\Kbar},\CF_V)&\to\\
\to (H^1(Y_{U',\Kbar},\Fl)\otimes V)^G&\to H^2(G,H^0(Y_{U',\Kbar},\Fl)\otimes V).\end{align*}
By Lemma~\ref{eisenstein}(a), the action of $G_K$ on the
first and last terms factors through an abelian quotient.
It follows that $\rho$ is modular of weight $V$ if and only if
$\rho$ is an $\Flbar[G_K]$-subquotient of
$H^1(Y_{U,\Kbar},\CF_V)\otimes\mu_\ell$.
Finally note that by Lemma~\ref{fctk}, we have 
$$H^1(Y_{U,\Kbar},\CF_{V(1)}) \cong H^1(Y_{U,\Kbar},\CF_V)\otimes\mu_\ell,$$
and part~(a) of the lemma follows.

(b) If $0\to W_1\to W_2\to W_3\to 0$ is a short exact sequence
of finite-dimensional $\Flbar[G]$-modules then
$0\to\CF_{W_1}\to\CF_{W_2}\to\CF_{W_3}\to0$ is a short exact sequence
of \'etale sheaves on $Y_U$, and (b) now follows from (a) and
Lemma~\ref{eisenstein}(a).

\end{proof}

Our chosen embeddings $\ol{\Q} \to \C$ and
$\ol{\Q} \to \ol{\Q}_\ell$
and identification of the residue field of $\ol{\Z}_\ell$ with $\ol{\F}_\ell$
allow
us to identify $S_K$ with $\bigcup_{\gp|\ell} S_\gp$ (because we are
assuming that $\ell$ is unramified in~$K$).
We now recall how this notion of modularity is related to the
existence of automorphic representations for $\GL_2/K$
giving rise to $\rho$. We start by establishing some conventions. 

When associating Galois representations to automorphic representations
we shall follow Carayol's conventions in~\cite{carayol:ens}. In
particular our normalisations of local and global class field theory
will send geometric Frobenius elements to uniformisers and our local-global
compatibility will be Hecke's rather than Langlands' (the one
that preserves fields of definition rather than the one that
behaves well under all functorialities; the difference is a dual
and a twist). We summarise Carayol's theorem (in fact we only
need a weaker form which is essentially due to Langlands and Ohta),
and its strengthening by Taylor and Jarvis.

For $k\geq2$ and $w$ integers of the same parity let $D_{k,w}$
denote the discrete series representation $D_{k,w}$ of $\GL_2(\R)$
with central character $t\mapsto t^{-w}$ defined in section~0.2
of~\cite{carayol:ens}. For $k=1$ and $w$ an odd integer we define
$D_{1,w}$ to be the irreducible principal series $\Ind(\mu,\nu)$
where the induction is unitary induction and
$\mu$ and $\nu$ are the (quasi-)characters of $\R^\times$
defined by $\mu(t)=|t|^{-w/2}\sgn(t)$ and $\nu(t)=|t|^{-w/2}$.
Now for $\vec{k} \in \Z^{S_K}$ with each $k_\tau \ge 1$ and
of the same parity, and $w\in\Z$ of the same parity as the $k_\tau$,
let us say that a cuspidal automorphic representation $\pi$ of $\GL_2(\A_K)$
is \emph{holomorphic of weight $(\vec{k},w)$} (or ``has weight $(\vec{k},w)$''
for short), if $\pi_{\tau}\cong D_{k_\tau,w}$
for each $\tau\in S_K$ (of course many cuspidal automorphic representations
will not have any weight---we are picking out the ones that correspond
to holomorphic Hilbert modular forms). The theorem of Eichler, Shimura,
Deligne, Deligne--Serre, Langlands, Ohta,
Carayol, Taylor, Blasius--Rogawski,
Rogawski--Tunnell and Jarvis associates a Galois representation
$\rho_\pi:G_K\to\GL_2(\Qlbar)$ to a cuspidal automorphic representation
$\pi$ for $\GL_2/K$ of weight $(\vec{k},w)$, and the correspondence
$\pi\mapsto\rho_\pi$ satisfies
Hecke's local-global compatibility at the finite places of~$K$
of characteristic not dividing~$\ell$ for which $\pi$ is unramified.
We remark that local-global compatibility at the ramified places away
from $\ell$ does not quite appear to be known in this generality,
although Carayol and Taylor establish it if $k_\tau\geq2$ for all $\tau$.
See Theorem~7.2 of~\cite{jarvis:wt1} for the
current state of play if $k_\tau=1$ for some $\tau$.

From this compatibility one deduces easily
(see section~3 of~\cite{carayol:ens} for example)
that if $\pi$ is holomorphic of weight $(\vec{k},w)$ then the determinant 
of $\rho_\pi$ is the product of a finite order
character and $\omega^{w-1}$, where $\omega$
denotes the cyclotomic character. Note that twisting by an appropriate
power of the norm character gives bijections between the automorphic
representations of weight $(\vec{k},w)$ and $(\vec{k},w+2n)$ for
any integer $n$; this corresponds to twisting by an appropriate
power of the cyclotomic character on the Galois side.

Let $\overline{\rho}_\pi:G_K\to\GL_2(\Flbar)$ denote the semisimplification of
the reduction of $\rho_\pi$. Our goal now in this section is to relate
two notions of being modular: the first is the notion of
being modular of some weight (in the sense
introduced above), and the second is the
notion of being modular in the sense of being isomorphic
to $\overline{\rho}_\pi$ for some holomorphic $\pi$.

\begin{proposition}\label{prop:classical}

Let $(\vec{k},w)\in\Z_{\geq2}^{S_K}\times\Z$ be integers all of the
same parity. Suppose $\rho:G_K\to\GL_2(\ol{\F}_\ell)$ is continuous,
irreducible and totally odd. Then $\rho\sim\ol{\rho}_\pi$
for some cuspidal automorphic representation $\pi$
for $\GL_2/K$ of weight $(\vec{k},w)$ and level prime to $\ell$
if and only if $\rho$ is modular of weight $V$ for some Jordan-H\"older
constituent $V$ of
$$V_{\vec{k},w}:=\bigotimes_{\gp|\ell} \bigotimes_{\tau\in S_\gp}  
\det{}^{(w-k_\tau)/2}\mathrm{Symm}^{k_\tau-2}k_\gp^2\otimes_\tau\ol{\F}_\ell.$$

\end{proposition}
\begin{remark} Note that the representation $V_{\vec{k},w}$ above
differs from the representation $\psi$ in section~2.1.1 of~\cite{carayol:ens}
by a twist due to the fact that we are using Jacobians
rather than \'etale cohomology.
\end{remark}
\begin{proof}
Say $V$ is a Jordan-H\"older constituent of $V_{\vec{k},w}$,
and that $\rho$ is modular of weight~$V$. Then, by definition,
there is a quaternion algebra $D/K$ satisfying the usual conditions,
and a level structure $U$ prime to $\ell$
such that $\rho$ is a subquotient of $(\Pic^0(X_{U'})[\ell]\otimes V)^G$
(with notation as above). By Lemma~\ref{kevinscontribution}(a) and (b)
and the remarks before Lemma~\ref{eisenstein},
$\rho$ is a $G_K$-subquotient of
$H^1(Y_{U,\Kbar},\CF^0_{\lambda,U}/\lambda)\otimes\Flbar$, where
$\CF^0_{\lambda,U}$ is the sheaf that Carayol associates to $(\vec{k},w)$.
Recall that by results of Jacquet-Langlands and Carayol~\cite{carayol:ens},
$H^1_p(Y_{U,\Kbar},\CF_{\lambda,U})$ is
a direct sum of irreducible 2-dimensional $\ell$-adic representations $\rho_\pi$
for $\pi$ as in the statement of the proposition. So by
Lemma~\ref{eisenstein} and a standard cohomological argument,
we see that $\rho\cong\overline{\rho}_\pi$ for some cuspidal automorphic
representation $\pi$ on $\GL_2/K$.

The reverse implication is not quite so straightforward because
given $\pi$ one needs to see $\overline{\rho}_\pi$ in the cohomology
of a Shimura curve (the problem being that $K$ might have even degree
and $\pi$ might be principal series at all finite places).
However, arguments of Wiles and Taylor show that this problem is surmountable
via ``level-raising.'' Indeed, Theorem~1 of~\cite{rlt:inv} and the
remarks following it show that $\overline{\rho}_\pi\cong\overline{\rho}_{\pi'}$
for $\pi'$ an automorphic representation of $\GL_2/K$ that is
special at a finite place, and the Galois representation associated
to $\pi'$ does indeed show up in the cohomology of a Shimura
curve by the above-mentioned results of Jacquet-Langlands and Carayol.
\end{proof}

\begin{corollary} If $\rho:G_K\to\GL_2(\Flbar)$ is continuous,
totally odd, irreducible, and $\rho\cong\overline{\rho}_\pi$
for some automorphic representation $\pi$ of $\GL_2/K$
of level prime to $\ell$ and weight
$(\vec{k},w)\in\Z_{\geq1}^{S_K}\times\Z$, then $\rho$ is
modular of weight~$V$ for some weight~$V$.
\end{corollary}
\begin{proof} If $k_\tau\geq2$ for all $\tau$ then this follows
immediately from the previous proposition. If $k_\tau=1$ for one
or more $\tau$ then it suffices to construct an automorphic
representation $\pi'$ of level prime to $\ell$
with $\overline{\rho}_{\pi}\cong\overline{\rho}_{\pi'},$
and with $\pi'$ of weight $(\vec{k}',w)$ with $k'_\tau\geq2$ for all $\tau$.
This is done in~\cite{jarvis:wt1} (via multiplication by an appropriate
modular form congruent to~1 mod~$\ell$: in particular the result
follows from the Deligne--Serre lemma and Lemma~5.2 of {\it loc.\ cit.}).
\end{proof}

We can furthermore predict the local behaviour at primes over $\ell$
of the automorphic representations of weight $(\vec{k},w)$ giving rise
to $\rho$. Before we start on this, here is a simple lemma that will
be of use to us later.

\begin{lemma}\label{lifting} If $k$ is a finite field of characteristic
$\ell$ and if $V$ is an irreducible $\Flbar$-representation of $\GL_2(k)$
then there is an irreducible $\Qlbar$-representation of $\GL_2(k)$
whose reduction has $V$ as a Jordan-H\"older factor. Furthermore
there is a 1-dimensional $\Flbar$-representation $\chi$ of $\GL_2(k)$
and an irreducible $\Qlbar$-representation of $\GL_2(k)$ with
a fixed vector for the subgroup $\smallmat{*}{*}{0}{1}$ whose
reduction has $\chi\otimes V$ as a Jordan-H\"older factor.
\end{lemma}
\begin{proof}
For 1-dimensional $V$ the result is clear (use the Teichm\"uller lift)
and for $V$ of dimension equal to the size of $k$, the Steinberg representation
does the job. For other $V$ the lemma follows immediately
from Proposition~1.1 of~\cite{fred:durham}
(with $J=S$ in the notation of that paper).
\end{proof}

We now introduce the following rather naive version of a type.
If $L$ is a finite extension of $\Q_\ell$, with integers $\CO_L$,
if $\pi$ is a smooth irreducible complex representation of $\GL_2(L)$
and if $\sigma$ is a smooth irreducible representation
of $\GL_2(\CO_L)$ (so $\sigma$ is finite-dimensional and its kernel
contains an open subgroup of $\GL_2(\CO_L)$), then we say that
$\pi$ is \emph{of type $\sigma$}
if the restriction of $\pi$ to a representation of $\GL_2(\CO_L)$
contains a subspace isomorphic to $\sigma$ (note that this is a much
weaker and simpler version of the usual notion of a type).

We now need a mild refinement of a level-raising result of Richard Taylor;
unfortunately this refinement does not appear to be in the literature,
so we sketch a proof here. The reader who wants to follow the details
is advised to have a copy of Taylor's paper~\cite{rlt:inv} handy.

\begin{lemma}\label{taylor} Suppose that $[K:\Q]$ is
even and $\pi$ is a cuspidal
automorphic representation of $\GL_2/K$ which has weight $(\vec{k},w)$.
Suppose moreover that $k_\tau\geq2$ for all infinite places $\tau$.
For all primes $\gp\mid\ell$ choose a smooth irreducible
representation $\sigma_\gp$ of $\GL_2(\CO_{K_\gp})$ such that
$\pi_\gp$ is of type $\sigma_\gp$.

Then there is a prime $\gq$ of $K$ not dividing $\ell$ and
a cuspidal automorphic representation $\pi'$ of $\GL_2/K$, also
of weight $(\vec{k},w)$, and such that
\begin{itemize}
\item $\ol{\rho}_\pi \sim \ol{\rho}_{\pi'}$;
\item $\pi'_\gq$ is an unramified special representation;
\item $\pi'_\gp$ is of type $\sigma_\gp$ each $\gp\mid\ell$.
\end{itemize}
\end{lemma}

\begin{proof} (sketch). The main idea of the argument is contained
in the proof of Theorem~1 of~\cite{rlt:inv}, but it would be a little
disingenuous to cite this result without further comment because,
as written, the proof does not keep track of types. It is not difficult
to change the argument so that it does, however. We indicate what needs
to be changed in order to prove the result we want. Because of our assumptions
about the weight of $\pi$ and the degree of $K$, the lemma
can be deduced from a purely combinatorial statement about automorphic forms
for the group of units of the totally definite quaternion algebra~$D$ over $K$
of discriminant~1. We need to check that the system of eigenvalues
associated to $\pi_D$ (the transfer of $\pi$ to $D^\times$)
is congruent to the system of eigenvalues associated
to an automorphic representation $\pi'_D$ which is Steinberg at some
place $\gq$ (this much is done in~\cite{rlt:inv}) and furthermore
such that $\pi'_{D,\gp}$ has type $\sigma_\gp$ at all places above $\ell$.
We do this by mimicking Taylor's argument with the following changes.
Instead of working at level $\Gamma_1(n)$ as in~\cite{rlt:inv},
we work with a more general compact open level structure $U$, assumed for
simplicity to be a product of local factors $U_\gp\subset\GL_2(K_\gp)$,
for $\gp$ running over the finite places of $K$. For $\gp\mid\ell$
we further assume that $U_\gp$ is a normal subgroup of $\GL_2(\CO_\gp)$,
with $\CO_\gp$ the integers in $K_\gp$ (all this can be achieved
by shrinking $U$ if necessary). We define $G_\ell$ to
be the finite group $\prod_{\gp\mid\ell}(\GL_2(\CO_\gp)/U_\gp)$;
the group $G_\ell$ then acts on the space of automorphic forms of level $U$.
Let $\sigma=\otimes_{\gp\mid\ell}\sigma_\gp$, so $\sigma$ is a
finite-dimensional
smooth complex irreducible representation
of $G_\ell$, and fix a number field $N$ which contains $F$, splits $D$,
contains the trace of $\sigma(g)$ for all $g\in G_\ell$ and furthermore
contains the values of the (algebraic) central character $\psi$ of $\pi_D$.
As in Taylor's paper we define $S_{(\vec{k},w)}^D(U)=S_{(\vec{k},w)}^D(U;\C)$,
the finite-dimensional complex vector space of weight~$(\vec{k},w)$ automorphic
forms of level $U$ for $D^\times$. Note
that our $\vec{k}$ is Taylor's $k$ and our $w$ is his $\mu$, and
that strictly speaking Taylor only
considers the case $w=\max\{k_\tau-2\}$, but his arguments never assume
this. We define
$S_{(\vec{k},w)}^D(U)_\psi$
to be the subspace of $S_{(\vec{k},w)}^D(U)$ where the centre~$Z$ of
$D^\times(\A_K)$
acts via the character~$\psi$. Note that for a fixed $U$ and $(\vec{k},w)$
there are only finitely many characters $\psi$ for which $S_{(\vec{k},w)}^D(U)_\psi$
is non-zero by finite-dimensionality of $S_{(\vec{k},w)}^D(U)$.
Furthermore the infinity type of such a $\psi$
is determined by~$w$; the character $\psi$ is an analogue of the
Dirichlet character associated to a classical modular form. For $R$ a
subring of $\C$ containing the integers of $N$ we define $S_{(\vec{k},w)}^D(U;R)$
as in Taylor's paper, and set
$S_{(\vec{k},w)}^D(U;R)_\psi=S_{(\vec{k},w)}^D(U;R)\cap S_{(\vec{k},w)}^D(U)_\psi$.
These spaces all have
an action of $G_\ell$;
we define $S_{(\vec{k},w)}^D(U)_{\sigma,\psi}$ to be the $\sigma$-eigenspace of
$S_{(\vec{k},w)}^D(U)_\psi$ (that is, the $\C[G_\ell]$-direct summand of $S_{(\vec{k},w)}^D(U)_\psi$
cut out by the idempotent in $\C[G_\ell]$ corresponding to $\sigma$),
and we define $S_{(\vec{k},w)}^D(U;R)_{\sigma,\psi}$ to be
$S_{(\vec{k},w)}^D(U;R)\cap S_{(\vec{k},w)}^D(U)_{\sigma,\psi}$.
Note that the natural projection map
$S_{(\vec{k},w)}^D(U)\to S_{(\vec{k},w)}^D(U)_{\sigma,\psi}$
typically does not induce a map
$S_{(\vec{k},w)}^D(U;R)\to S_{(\vec{k},w)}^D(U;R)_{\sigma,\psi}$
because of denominators, but some positive integer multiple of it will do.
More precisely,
define $N(U)$ to be the product of
the order of $G_\ell$ and the order of the finite group
$K^\times\det(U)\backslash\A_K^\times/(K_\infty)^o$.
The composite of the projection map above and multiplication
by $N(U)$ induces a map
$e:S_{(\vec{k},w)}^D(U;R)\to S_{(\vec{k},w)}^D(U;R)_{\sigma,\psi}$
such that the composite
$S_{(\vec{k},w)}^D(U;R)_{\sigma,\psi}\subseteq S_{(\vec{k},w)}^D(U;R)\to S_{(\vec{k},w)}^D(U;R)_{\sigma,\psi}$
is multiplication by $N(U)$. Note crucially
that $N(U)=N(U\cap U_0(\gq))$, and that $e^2=N(U)e$ as endomorphisms
of $S_{(\vec{k},w)}^D(U;R)$.

Contrary to what Taylor implicitly asserts,
there is in general no $\SL_2(R)$-invariant perfect pairing on $\Symm^a(R^2)$
if $R$ is not a $\Q$-algebra.
However if $R$ is a subring of $\C$ then there is an
$\SL_2(R)$-invariant injection from $\Symm^a(R^2)$ to its $R$-dual, with
cokernel killed by some positive integer $C'$ (which depends on $a$
but not on $R$),
and this induces a pairing on $\Symm^a(R^2)$ which is not perfect but
which will suffice to prove the result we need (some of Taylor's constants
need to be modified by this constant). Taylor uses this pairing to
produce a perfect pairing on $S_{(\vec{k},w)}^D(U;\C)$ and the analogue of this
pairing that we shall need is the induced perfect pairing
between $S_{(\vec{k},w)}^D(U;\C)_{\sigma,\psi}$ and
$S_{(\vec{k},w)}^D(U;\C)_{\sigma^*,\psi^*}$, where $\sigma^*=\sigma(\chi\circ\det)$
and $\psi^*=\psi(\chi\circ\det)$, where $\chi$ is the finite order Hecke
character associated to $\pi_D$ on p.272 of~\cite{rlt:inv} and $\det$
is the reduced norm $D^\times\to\GL_1$. The reason for this twist
is that the $\SL_2$-invariant pairing on the coefficient sheaves
is not $\GL_2$-invariant.

We now run through Taylor's argument on pp.272--276 of {\it loc.\ cit.},
making the following changes.
If $S_k^D(U_1(n);R)$ occurs on the left hand side of a pairing, we replace
it by $S_{(\vec{k},w)}^D(U;R)_{\sigma,\psi}$; if it occurs on the right hand side
then we replace it by $S_{(\vec{k},w)}^D(U;R)_{\sigma^*,\psi^*}$.
We replace $U_1(n;\gq)$ with $U\cap\Gamma_0(\gq)$ (with $\Gamma_0(\gq)$
denoting the usual level structure, namely the matrices which are
upper triangular modulo $\gq$), and replacing the Hecke algebra
$\T_k^D(n)$ in~\cite{rlt:inv} by the sub-$\Z$-algebra
$\T_{(\vec{k},w)}(U)_{\sigma,\psi}$ of
$\End_{\C}(S_{(\vec{k},w)}^D(U)_{\sigma,\psi})$ generated
only by Hecke operators $T_{\gq}$ at the \emph{unramified} primes~$\gq$.
These Hecke algebras are as big as we shall need for
our application---we do not need to consider the operators $S_{\gq}$ as we
have fixed a central character, and we also do not need to consider Hecke
operators at the ramified places.

The analogues of the assertions about direct sums in Lemma~1 of~\cite{rlt:inv}
are still true on the $(\sigma,\psi)$-component of $S_{(\vec{k},w)}^D(U)$.
Taylor's map $i$ commutes with the action
of $G_\ell$ and with the action of $Z$, 
and the analogue of Taylor's Lemma~2 holds (indeed the proof given remains
valid when $S_k^D(n)$ is replaced by $S_{(\vec{k},w)}^D(U)_{\sigma,\psi}$ etc.,
as the map $i^\dag$ also commutes with the $G_\ell$ and $Z$-action).
The analogue of Lemma~3 that we need is that for a fixed
compact open subgroup $X\subset\GL_1(\A_K^f)$ there are positive integer
constants $C_1$
and $C_2$ such that for any compact open $U$ with $\det(U)=X$
we have
$$C_1\langle S_{(\vec{k},w)}^D(U;R)_{\sigma,\psi},
S_{(\vec{k},w)}^D(U;R)_{\sigma^*,\psi^*}\rangle\subseteq R$$
and for $f\in S_{(\vec{k},w)}^D(U;\C)_{\sigma,\psi}$ with
$$\langle f,S_{(\vec{k},w)}^D(U;R)_{\sigma^*,\psi^*}\rangle\subseteq R$$
we have $C_2 f\in S_{(\vec{k},w)}^D(u;R)_{\sigma,\psi}$.
The statement about $C_1$ follows from Taylor's Lemma~3, and that
about $C_2$ can also be deduced from Taylor's result, the comments
about $C'$ above, the fact
that the pairing $\langle,\rangle$ restricts to the zero pairing
between the $(\sigma,\psi)$-eigenspace and the $(\sigma',\psi')$-eigenspace
if $(\sigma',\psi')\not=(\sigma^*,\psi^*)$ (using the argument on the bottom
of p.272 of~\cite{rlt:inv}) and the existence of the ``projector'' $e$
above. 

We need the $(\sigma^*,\psi^*)$-analogue of Lemma~4 of~\cite{rlt:inv}
and this is true---indeed it can be deduced from Lemma~4 of~\cite{rlt:inv}
by restricting to the $(\sigma^*,\psi^*)$-eigenspace.

We can now prove the analogue of Theorem~1 of~\cite{rlt:inv} (where
we replace Taylor's Hecke algebras with ours as indicated above); 
we simply mimic Taylor's beautiful proof on p.276 of {\it loc.\ cit.}; the
assiduous reader can check that we have explained the analogues of
all the ingredients that we need. Now using a standard
Cebotarev argument we deduce that given $\pi_D$
of weight $(\vec{k},w)$, we can find a prime $\gq\nmid\ell$ of $K$ at which
$\pi_D$ is unramified principal series and such that
the associated system of eigenvalues
of $\T_{(\vec{k},w)}(U)_{\sigma,\psi}$ is congruent (modulo some prime above $\ell$)
to a system of eigenvalues occurring in
$\T_{(\vec{k},w)}(U_0(\gq))^{\new}_{\sigma,\psi}$.
\end{proof}

Using this beefed-up version of Taylor's level-raising theorem, we
can deduce a beefed-up version of Proposition~\ref{prop:classical}.
Suppose that $\vec{k} \in \Z^{S_K}$ and $w\in\Z$ with 
$k_\tau \ge 2$ and of the same parity as $w$ for all $\tau$.
For each $\gp|\ell$, suppose that $\sigma_\gp$ is a smooth irreducible
representation of $\GL_2(\CO_{K,\gp})$. Via our fixed embeddings
$\Qbar\to\C$ and $\Qbar\to\Qlbar$, and our identification of the residue
field of $\Qlbar$ with $\Flbar$, we can unambiguously define
the semisimplification $\overline{\sigma}_\gp$ of the mod $\ell$
reduction of $\sigma_\gp$; so $\overline{\sigma}_\gp$ is a representation
of $\GL_2(\CO_{K,\gp})$ on a finite-dimensional $\Flbar$-vector space.
Define $\sigma:=\otimes_{\gp\mid\ell}\sigma_\gp$ and 
$\overline{\sigma}:=\otimes_{\gp\mid\ell}\overline{\sigma}_{\gp}$.
Finally let $G_\ell$ denote a finite quotient of $\GL_2(\CO_K\otimes\Z_\ell)$
through which $\sigma$ factors.

\begin{proposition} \label{prop:types}
For an irreducible representation $\rho:G_K\to\GL_2(\Flbar)$, the
following are equivalent:
\begin{itemize}
\item $\rho \sim \ol{\rho}_\pi$
for some cuspidal holomorphic weight $(\vec{k},w)$
automorphic representation $\pi$ of $\GL_2(\A_K)$
such that $\pi_\gp$ has type $\sigma_\gp$ for each $\gp|\ell$.
\item $\rho$ is modular of weight $V$ for some Jordan-H\"older constituent $V$ of
$$W:=\overline{\sigma}^\vee\otimes \bigotimes_{\gp\mid\ell}\bigotimes_{\tau\in S_\gp}\left(
\det{}^{(w-k_\tau)/2}\mathrm{Symm}^{k_\tau-2}k_\gp^2\otimes_\tau\ol{\F}_\ell\right).$$
\end{itemize}
\end{proposition}
\begin{proof} 

Lemma~\ref{taylor} shows that $\rho\sim\ol{\rho}_\pi$ for some $\pi$
as above if and only if there are $D$ and $U$ as usual, and a compact open subgroup $U''$
of $U$ with $U/U''=G_\ell$, such that 
$\rho$ is the mod $\ell$ reduction of an irreducible
2-dimensional $\Qlbar$-representation
$\tilde{\rho}$ which is a direct summand (equivalently, a subquotient)
of $\Hom_{G_\ell}(\sigma,H^1_p(Y_{U'',\overline{K}},\CF_{\lambda,U''}))$,
with $\CF_{\lambda,U''}$ the $E_\lambda$-sheaf constructed by
Carayol on $X_{U''}$ corresponding to $(\vec{k},w)$ (where here
we assume $U$ is small enough for the projection
$\phi:Y_{U''}\to Y_U$ to be \'etale with covering group $G_\ell$,
and furthermore that our coefficient field $E_\lambda$
is assumed large enough to satisfy the conditions
that Carayol requires of it, and also large enough to afford a model
for $\sigma$ and to ensure that all the irreducible subquotients of
$H^1(Y_{U'',\overline{K}},\CF_{\lambda,U''})$ have dimension either~1 or~2).
By Lemma~\ref{eisenstein}(b) the latter condition on $\rho$ is equivalent to
$\rho$ being isomorphic to the mod $\ell$ reduction of an irreducible
2-dimensional subquotient of
$\Hom_{G_\ell}(\sigma,H^1(Y_{U'',\overline{K}},\CF_{\lambda,U''}))$.

Now let
let $\CF_{\sigma^{\vee}}$
denote the $E_\lambda$-sheaf on $Y_{U}$ associated to the dual of $\sigma$,
and let $\CF_{\lambda,U}$ denote the $E_\lambda$-sheaf constructed by Carayol
on $Y_{U}$ and corresponding to weight $(\vec{k},w)$.
Then
\begin{align*}
\Hom_{G_\ell}(\sigma,H^1(Y_{U'',\overline{K}},\CF_{\lambda,U''}))&=(H^1(Y_{U'',\overline{K}},\CF_{\lambda,U''})\otimes\sigma^{\vee})^{G_\ell}\\
&=H^1(Y_{U'',\overline{K}},\CF_{\lambda,U''}\otimes \phi^*\CF_{\sigma^\vee})^{G_\ell}\\
&=H^1(Y_{U'',\overline{K}},\phi^*(\CF_{\lambda,U}\otimes\CF_{\sigma^\vee}))^{G_\ell}\\
&=H^1(Y_{U,\overline{K}},\CF_{\lambda,U}\otimes\CF_{\sigma^\vee}),
\end{align*}
the last equality coming, for example, from the Hochschild-Serre spectral
sequence and the fact that the order of $G_\ell$ is invertible in the
(characteristic zero) field $E_\lambda$.

Let $\CO_\lambda$ denote the integers in $E_\lambda$, let
$\CF^0_{\lambda,U}$ and $\CF^0_{\sigma^\vee}$ be $\CO_\lambda$-lattices
in $\CF_{\lambda,U}$ and $\CF_{\sigma^\vee}$ and set
$\CF^0:=\CF^0_{\lambda,U}\otimes\CF^0_{\sigma^\vee}$. We then deduce
that an irreducible $\rho$ is the reduction of an irreducible
$\overline{\rho}_\pi$ as above if and only if it is a subquotient
of $H^1(Y_{U,\overline{K}},\CF^0)^{tf}/\lambda$,
where $\lambda$ denotes the maximal ideal of $\CO_\lambda$ and $tf$ denotes
the maximal torsion-free quotient. Now a standard argument shows
that this is so if and only if $\rho$ is a subquotient of $H^1(Y_{U,\overline{K}},\CF^0/\lambda)$
(the torsion in $H^1(Y_{U,\overline{K}},\CF^0)$
is a subquotient of $H^0(Y_{U,\overline{K}},\CF^0/\ell^N)$ for some $N$ and
hence the Galois representations arising as subquotients of
it are all 1-dimensional by Lemma~\ref{eisenstein}(a),
and the cokernel of the injection
$H^1(Y_{U,\overline{K}},\CF^0)/\lambda\to H^1(Y_{U,\overline{K}},\CF^0/\lambda)$ is
contained in $H^2(Y_{U,\overline{K}},\CF^0)[\lambda]$ and hence in $H^2(Y_{U,\overline{K}},\CF^0/\lambda^N)$
for some $N$, and the irreducible subquotients of this group are also
all 1-dimensional by Lemma~\ref{eisenstein}(a)). To finish the proof
it suffices by Lemma~\ref{kevinscontribution}
to check that $\CF^0/\lambda\otimes\Flbar$ and $W(1)$
have the same Jordan-H\"older factors, which follows
immediately from the definitions.
\end{proof}

For $\tau \in S_\gp$, we denote by $\omega_\tau$
the fundamental character of $I_{K_\gp}$ defined by composing $\tau$
with the homomorphism $I_{K_\gp} \to k_\gp^\times$ obtained
from local class field theory (with the convention that uniformizers
correspond to geometric Frobenius elements).
We then have the following compatibility among determinants,
central characters and twists.
\begin{corollary}\label{gl1}
\begin{enumerate} 
\item 
If $\rho$ is modular of weight $V$ and $V$ has central character $\otimes_{\gp|\ell}
\prod_{\tau\in S_\gp} \tau^{c_\tau}$,
then 
$$\det\rho|_{I_{K_\gp}}= \prod_{\tau\in S_\gp}\omega_\tau^{c_\tau+1}$$
for each $\gp|\ell$.
\item 
Let $\chi:G_K \to \ol{\F}_\ell^\times$ be such that
$\chi|_{I_{K_\gp}}= \prod_{\tau\in S_\gp}\omega_\tau^{c_\tau}$
for each $\gp|\ell$. Then $\rho$ is modular of weight $V$ if and only 
$\chi\rho$ is modular of weight $V\otimes V_\chi$, where
$$V_\chi = \bigotimes_{\gp|\ell} \bigotimes_{\tau\in S_\gp} \det{}^{c_\tau}
k_\gp^2\otimes_\tau\ol{\F}_\ell.$$
\end{enumerate}
\end{corollary}
\begin{proof}
(1) Let $T$ denote the Teichm\"uller lift of the norm
map $N:(\CO/\ell)^\times\to\Fl^\times$. By Lemma~2.7 there is an irreducible
$\Qlbar$-representation $\sigma$ of $G$ such that the reduction of
$\sigma^\vee\otimes(T^{-1}\circ\det)$ contains $V$ as a Jordan-H\"older factor.
By Proposition~\ref{prop:types} we see that if $\rho$ is modular
of weight~$V$ then $\rho\sim\overline{\rho}_\pi$ for some $\pi$
of weight $(\vec{2},0)$ and type $\sigma$. Now by section 5.6.1
of~\cite{carayol:ens} we see that $\det(\rho_\pi)=\chi_\pi^{-1}\omega^{-1}$
where $\chi_\pi$ is the central character of $\pi$ and $\omega$
is the cyclotomic character.
Now $\chi_\pi$ can be computed on $\CO_{K_\gp}^\times$ because it
is the central character of $\sigma_{\gp}$, which is the inverse
of the central character $\alpha_\gp$ of $\sigma_{\gp}^\vee$. We deduce
that
$\det(\rho)|I_\gp=\alpha_\gp\overline{\omega}^{-1}=\alpha_{\gp}N_\gp^{-1}$,
where $N_\gp=\prod_{\tau\in S_\gp}$ is the map
$I_{K_\gp}\to k_{\gp}^\times\to\F_\ell$. Finally the fact that $V$
is a Jordan-H\"older factor of the reduction of
$\sigma^{\vee}\otimes(T^{-1}\circ\det)$ implies that the central
character of $V$ (considered as a representation of $\GL_2(k_\gp)$)
is $\alpha_\gp.N_\gp^{-2}$ and the result follows.

(2) is simpler and could have been deduced earlier; in fact
it is immediate from Lemma~\ref{fctk} and Lemma~\ref{kevinscontribution}(a).
\end{proof}

Recall that for $K = \Q$, every modular $\rho$ arises from a form
of level prime to $\ell$ and some weight $k \ge 2$. Moreover,
after twisting $\rho$, one may take the weight $k$ to be in the
range $2 \le k \le \ell+1$ (see \cite{as1}). This is in general false
for larger $K$. Indeed if $\rho$ arises from a form of weight $(\vec{k},w)$
and level prime to $\ell$, then we see from the preceding results
that $\det\rho|_{I_\gp} = \omega^{w-1}$
for all $\gp|\ell$, and it is easy to construct representations
none of whose twists have this property: choose for example an
odd prime $\ell$ inert in a real quadratic $K$, and a totally odd
$\rho$ such that $\det\rho|_{I_\ell} = \omega_\tau^a$ for some
odd integer $a$ (where $\tau:\CO_K/\ell \to \Flbar$);
see \cite{ddr} for some explicit examples.
On the other hand,
it is still the case that every modular $\rho$ arises from a form of weight 
$(2,\ldots,2)$ and some level not necessarily prime to $\ell$.
Moreover, after twisting $\rho$, we can assume the form
has level dividing $\gn\ell$ for some $\gn$ prime to $\ell$.

\begin{corollary}\label{cor:equiv}
For an irreducible $\rho:G_K\to\GL_2(\Flbar)$
the following are equivalent:
\begin{enumerate}
\item $\rho \sim \ol{\rho}_\pi$ for some holomorphic cuspidal
automorphic representation $\pi$ of $\GL_2(\A_K)$;
\item $\rho \sim \ol{\rho}_\pi$ for some cuspidal automorphic
representation $\pi$ of $\GL_2(\A_K)$ which is holomorphic
of weight $(\vec{2},0)$;
\item $\chi\rho\sim\ol{\rho}_\pi$ for some character $\chi$ and some
cuspidal automorphic representation $\pi$ of $\GL_2(\A_K)$
which is holomorphic of weight $(\vec{2},0)$ and level $U=U^\ell.U_1(\ell)$ (the
adelic analogue of ``level $\Gamma_1(\ell)$ at $\ell$'');
\item $\rho$ is modular of weight $V$ for some Serre weight $V$.
\end{enumerate}
\end{corollary}
\begin{proof} It is clear that
(3) $\Rightarrow$ (2) $\Rightarrow$ (1). Proposition~\ref{prop:types}
shows that (1) $\Rightarrow$ (4) if $k_\tau\geq2$ for all $\tau$;
if some of the $k_\tau$ are equal to~1 then one has to first
multiply by an Eisenstein series to increase the weight $\vec{k}$:
more formally one uses the Deligne--Serre lemma and
Lemma~5.2 of~\cite{jarvis:wt1}.
Finally we need to show that (4) $\Rightarrow$ (3). If (4) holds
then by Lemma~\ref{lifting} there is an irreducible
$\Qlbar$-representation $\sigma$ of $G$ with a $\smallmat**01$-fixed
vector such that the reduction of $\sigma^\vee\otimes\bigotimes_{\gp|\ell}
\otimes_{\tau\in S_\gp}\det^{-1}$ contains some twist of $V$.
We deduce (3) from Lemma~\ref{gl1}(2) and the case $(\vec{k},w)=(\vec{2},0)$
of \ref{prop:types}.
\end{proof}

\section{The weight conjecture}
\label{sec:recipe}

Suppose that $\rho:G_K \to \GL_2(\ol{\F}_\ell)$ is continuous,
irreducible and totally odd. 
The aim of this section is to provide a conjectural recipe
for the set of $V$ such that $\rho$ is modular of weight $V$.

For each prime $\gp$ of $K$ dividing $\ell$, we will define a set of representations
$W_{\gp}(\rho)$ of $\GL_2(k_\gp)$ depending only
on $\rho|_{I_{K_\gp}}$, and then define the conjectural weight set $W(\rho)$ 
as the set of Serre weights of the form $\otimes_{\ol{\F}_\ell} V_\gp$
with $V_\gp \in W_{\gp}(\rho)$.

We need some more notation before defining $W_\gp(\rho)$.
With our prime $\gp$ dividing $\ell$ fixed for now, we write simply
$k$, $f$ and $S$ for $k_\gp$, $f_\gp$ and $S_\gp$. Fix an embedding
$\ol{K} \to \ol{K}_\gp$ and identify $D=G_{K_\gp}$ and $I=I_{K_\gp}$
with subgroups of $G_K$. Let $K_\gp'$ be the unramified
quadratic extension of $K_\gp$ in $\ol{K}_\gp$ and let $k'$
denote its residue field. We let $S'$ denote
the set of embeddings $k' \to \ol{\F}_\ell$, let $D' = G_{K_\gp'}$
and define a map
$\pi:S'\to S$ by $\tau'\mapsto\tau'|_k$. 

Suppose that $L \subset \ol{K}_\gp$ is a finite unramified extension
of $K_\gp$ and $\sigma$ is an embedding of its residue field
$\CO_L/\ell\CO_L$ in $\ol{\F}_\ell$. We denote by $\omega_\sigma$
the fundamental character of $I = I_L$ defined by composing $\sigma$
with the homomorphism $I_L \to (\CO_L/\ell\CO_L)^\times$ gotten
from local class field theory.

In defining $W_\gp(\rho)$, we treat separately the cases where $\rho|_D$ is
irreducible and where it is reducible.

\subsection{The irreducible case.}

If $\rho|_D$ is irreducible,
we define $W_\gp(\rho)$ by the following rule:
\begin{equation}\label{eqn:irred}
\begin{array}{ccl} &&
\rho|_I\sim\prod_{\tau \in S}\omega_\tau^{a_\tau}
\begin{pmatrix}\prod_{\tau'\in J}\omega_{\tau'}^{b_{\pi(\tau')}}&0\\
0&\prod_{\tau'\notin J}\omega_{\tau'}^{b_{\pi(\tau')}}\end{pmatrix}\\
V_{\vec{a},\vec{b}}\in W_\gp(\rho) &\Longleftrightarrow& \\
&&\mbox{ for some $J\subset S'$ such that $\pi:J\iso S$.}\end{array}
\end{equation}

Since $\rho|_D$
is irreducible, there is a character $\xi:D'\to \ol{\F}_\ell^\times$ such
that $\rho|_D \sim \Ind_{D'}^D \xi$. We define 
$$W'(\xi) = \{\,(V_{\vec{a},\vec{b}},J)\,|\, J \subset S', \, \pi:J \iso S,\,
\xi|_I = \prod_{\tau \in S}\omega_\tau^{a_\tau}\prod_{\tau'\in J}\omega_{\tau'}^{b_{\pi(\tau')}}\,\}.$$
Thus $W_\gp(\rho) = \{\,V\,|\,(V,J)\in W'(\xi)\mbox{ for some $J$}\,\}$. (Note that replacing
$\xi$ by its conjugate under $D/D'$ replaces $J$ by its complement.)
We shall see that the projection maps $W'(\xi) \to W_\gp(\rho)$ and
$W'(\xi) \to \{\,J\subset S'\,|\pi: J \iso S\,\}$ are typically 
bijections, so that typically (but not always) $|W_\gp(\rho)| = 2^f$. 

We now choose an element of $S'$ which we denote $\tau_0'$,
and then let $\tau_i' = \tau_0'\circ \frob_\ell^i$ and $\tau_i = \pi(\tau_i')$.
Note that $S = \{\,\tau_i\,|\,i \in \Z/f\Z\,\}$ and $S' = \{\,\tau_i'\,|\,i \in \Z/2f\Z\,\}$.
Letting $\omega = \omega_{\tau_0}$
and $\omega' = \omega_{\tau_0'}$, we have $\omega_{\tau_i} = \omega^{\ell^i}$,
$\omega_{\tau_i'} = (\omega')^{\ell^i}$ and $\omega = (\omega')^{\ell^f + 1}$.
Note that $\xi|_I = (\omega')^n$ for some $n \bmod \ell^{2f} - 1$, and since
$\rho|_D$ is irreducible, $n$ is not divisible by $\ell^f +1$.

For $B \subset \{0,\ldots,f-1\}$ (where the symbol $\subset$ includes
the case of equality),
let $J_B = \{\tau'_i | i \in B\} \cup \{\tau'_{f+i}|i\not\in B\}$.
If $a\in \Z/(\ell^f-1)\Z$, $\vec{b} = (b_0,\ldots,b_{f-1})$ with each $b_i\in \{1,\ldots,\ell\}$
and $B\subset \{0,\ldots,f-1\}$, let
$$n'_{a,\vec{b},B} = a(\ell^f + 1) + \sum_{i\in B}
 b_i\ell^i + \sum_{i\not\in B} b_{i}\ell^{f + i} \bmod \ell^{2f}-1.$$
Then $W'(\xi)$ is in bijection with the set of triples $(a,\vec{b},B)$ as above
with $n \equiv n'_{a,\vec{b},B} \bmod \ell^{2f}-1$. Now note for each
$B \subset \{0,\ldots,f-1\}$, there is a unique such triple $(a,\vec{b},B)$
with this property for each
solution of
$$n \equiv \sum_{i\in B} b_i\ell^i - \sum_{i\not\in B} b_i\ell^i \bmod \ell^f+1$$
with $b_0,\ldots,b_{f-1} \in \{1,\ldots,\ell\}$.
But the values of $\sum_{i\in B}b_i\ell^i - \sum_{i\not\in B} b_i\ell^i$
are the $\ell^f$ consecutive integers
from $n'_B - \ell^f$ to $n'_B - 1$ where 
$$n'_B = \sum_{i\in B}\ell^{i+1} - \sum_{i\not\in B}\ell^i + 1,$$
so there is a solution as long as $n \not\equiv n_B' \bmod \ell^f + 1$ and this
solution is unique.
We have thus shown that $W'(\xi)$ is in bijection with the set of $B$ such
$n \not\equiv n_B' \bmod \ell^f + 1$; moreover the projection
$W'(\xi) \to \{\,J\subset S'\,|\pi: J \iso S\,\}$ is injective.

Note that if $f$ is odd and $B$ is either $\{0,2,4,\ldots,f-1\}$
or $B = \{1,3,5,\ldots,f-2\}$, then
$n'_B \equiv 0 \bmod \ell^f + 1$. To see that the converse holds as well, observe that
$$-(\ell^f+1) < -\ell\frac{\ell^{f-1}-1}{\ell-1} \le n'_B \le \frac{\ell^{f+1}-1}{\ell-1} < 2 (\ell^f + 1).$$
Thus if $n'_B \equiv 0 \bmod \ell^f + 1$, then $n'_B = 0$ or $\ell^f+1$.
If $n'_B = 0$, then solving 
$$\sum_{i\in B}\ell^{i+1} - \sum_{i\not\in B}\ell^i + 1 \equiv 0 \bmod \ell^r $$
by induction on $r$, we find that $f$ is odd and $B = \{1,3,\ldots,f-2\}$.
Similarly if $n'_B = \ell^f + 1$, then $f$ is odd and $B = \{0,2,\ldots,f-1\}$.

We now show that if $f$ is even, then the $2^f$ values of $n'_B\bmod \ell^f+1$
are distinct. First note that
$$n'_B \equiv -1 + (\ell + 1)\sum_{i\in B^*} (-1)^i\ell^i \bmod \ell^f + 1,$$
where $B^* = (\{0,2,\ldots,f-2\}\cap B) \cup (\{1,3,\ldots,f-1\}\setminus B)$.
Thus $n'_{B_1} \equiv n'_{B_2} \bmod \ell^f + 1$ if and only if
$$\sum_{i\in B_1^*} (-1)^i\ell^i \equiv \sum_{i\in B_2^*}(-1)^i\ell^i \bmod (\ell^f + 1)/d,$$
where $d = \gcd(\ell+1,\ell^f + 1) \le \ell-1$ (so $d=2$ if $\ell$ is odd, and
$d=1$ if $\ell=2$). But these two sums differ by
at most $(\ell^f - 1)/(\ell-1) < (\ell^f + 1)/d$, so the above congruence holds
if and only if equality holds, in which case $B_1^* = B_2^*$, so $B_1 = B_2$.

Next we show that if $f$ and $\ell$ are odd, then the $2^f -2$ non-zero
values of $n'_B\bmod \ell^f+1$
are distinct. In this case we have
$$n'_B \equiv (\ell + 1)\sum_{i\in B^*} (-1)^i\ell^i \bmod \ell^f + 1,$$
where $B^* = (\{0,2,\ldots,f-1\}\cap B) \cup (\{1,3,\ldots,f-2\}\setminus B)$.
Thus $n'_{B_1} \equiv n'_{B_2} \bmod \ell^f + 1$ if and only if
$$\sum_{i\in B_1^*} (-1)^i\ell^i \equiv \sum_{i\in B_2^*}(-1)^i\ell^i \bmod (\ell^f + 1)/(\ell+1).$$
But these two sums differ by
at most $(\ell^f - 1)/(\ell-1) < 2 (\ell^f + 1)/(\ell + 1)$, so if the above congruence holds
then either equality holds, in which case $B_1= B_2$, or 
$$\sum_{i\in B_2^*} (-1)^i\ell^i = \sum_{i\in B_1^*}(-1)^i\ell^i + \sum_{i=0}^{f-1}(-1)^i\ell^i,$$
exchanging $B_1$ and $B_2$ if necessary. Solving $\bmod\ \ell^r$ inductively on $r$,
we see that the only possibility is that $B_1^* = \emptyset$ and $B_2^* = \{0,1,\ldots,f-1\}$,
but these are precisely the cases where $n'_B \equiv 0 \bmod \ell^f+1$.

Finally suppose that $f$ is odd and $\ell = 2$. In this case we have 
$$n'_B \equiv 3\sum_{i\in B^*} (-1)^i 2^i \bmod 2^f + 1,$$
where $B^* = (\{0,2,\ldots,f-1\}\cap B) \cup (\{1,3,\ldots,f-2\}\setminus B)$.
In particular $n'_B \equiv 0 \bmod 3$. Moreover we have the inequality
$|n_{B_1}' - n_{B_2}'| < 3(2^f + 1)$, showing that each congruence class
$\mbox{mod $2^f+1$}$ arises as $n_B'$ for at most $3$ values of $B$. It follows
that each of the $(2^f-2)/3$ non-zero multiples of $3\bmod 2^f+1$ arises as $n'_B$
for exactly $3$ values of $B$. 

We have thus proved the following propositions:

\begin{proposition}\label{prop:irred} Suppose that $\ell$ is odd. If $f$ is even, then the
congruence classes $\bmod\ \ell^f + 1$ of the form
$$ -1 + (\ell + 1)\sum_{i\in B^*} (-1)^i\ell^i$$
are distinct and non-zero as $B^*$ runs through all subsets of $\{0,1,\ldots,f-1\}$.
 If $f$ is odd, then the
congruence classes $\bmod\ \ell^f + 1$ of the form
$$(\ell + 1)\sum_{i\in B^*} (-1)^i\ell^i$$
are distinct and non-zero as $B^*$ runs through all non-empty proper subsets of 
$\{0,1,\ldots,f-1\}$. Letting $A$ denote the set of such classes in each case, we have
$$|W'(\xi)| = \left\{\begin{array}{ll}
2^f , & \mbox{if $n\not\in A$,}\\
2^f - 1, & \mbox{if $n\in A$,}\end{array}\right.$$
where $\xi|_I=(\omega')^n$.
\end{proposition}

\begin{proposition} \label{prop:irred2} Suppose that $\ell = 2$
and $\xi|_I=(\omega')^n$. Then
$$|W'(\xi)| = \left\{\begin{array}{ll}
2^f-1, & \mbox{if $f$ is even,}\\
2^f, & \mbox{if $f$ is odd and $3\nmid n$,}\\
2^f-3, & \mbox{if $f$ is odd and $3\mid n$.}
\end{array}\right.$$
\end{proposition}

Multiple $B$ can occur with the same $(a,\vec{b})$; for example if $f=3$
and $n=1$, then $(-\ell^2,(1,\ell,1),\{0,1\})$ and $(-\ell^2,(1,\ell,1),\{0,2\})$ 
are both in $W'(\xi)$, so the map
$W'(\xi) \to W_\gp(\rho)$ is not injective. (In this case in fact,
$|W'(\xi)| = 8$, but $|W_\gp(\rho)| = 6$.) 
\begin{proposition} The map $W'(\xi) \to W_\gp(\rho)$ fails to be injective
if and only if $\ell^r n \equiv m\bmod \ell^f + 1$ for some integers $r, m$
with $|m| \le \ell(\ell^{f-2}-1)/(\ell - 1)$.
\end{proposition}
\begin{proof} Suppose first that $W'(\xi) \to W_\gp(\rho)$ is not injective.
This means that for some $a$, $\vec{b}$ and $B_1\neq B_2$, we have 
\begin{equation}\label{eqn:injfail}
n \equiv n'_{a,\vec{b},B_1} \equiv n'_{a,\vec{b},B_2}\bmod \ell^{2f}-1.\end{equation}

First we note that $B_2$ cannot be the complement $\ol{B}_1$
of $B_1$ in $\{0,\ldots,f-1\}$ since 
$$n'_{a,\vec{b},B_1}\equiv
n'_{a,\vec{b},\ol{B}_1} \equiv \ell^f n'_{a,\vec{b},B_1} \bmod \ell^{2f}-1$$
would imply $n \equiv 0 \bmod \ell^f + 1$, contradicting the
irreducibility of $\rho|_D$. We thus have $J_2 \neq S' \setminus J_1$,
where $J_1 = J_{B_1}$ and $J_2 = J_{B_2}$. One now checks
that (possibly after switching $J_1$ and $J_2$) we may find
$t \in \Z/2f\Z$ so that $\tau_{t-1}' \in J_1 \setminus J_2$ and
$\tau_t'\in J_1 \cap J_2$. We then have
$$\ell^{f-t}n \equiv n'_{a',\vec{b}',B_1'} \equiv n'_{a',\vec{b}',B_2'}\bmod \ell^{2f}-1,$$
where $a' \equiv \ell^{-t}a\bmod \ell^f-1$, $b_i' = b_{i+t\bmod f}$ for $i \in \{0,\ldots,f-1\}$
and $B_\nu'$ is such that $J_{B_\nu'} = J_\nu \circ \frob_\ell^{f-t}$ for $\nu = 1,2$.
Replacing $n$ with $\ell^{f-t}n$, we may thus assume that $\tau_{f-1}' \in J_1 \setminus J_2$
and $\tau_f' \in J_1 \cap J_2$, or equivalently, $f-1 \in B_1 \setminus B_2$ and
$0 \not\in B_1 \cup B_2$.

Returning to the congruence (\ref{eqn:injfail}), we have
$$\sum_{i\in B_1} b_i\ell^i + \ell^f \sum_{i\not\in B_1} b_i\ell^i
\equiv \sum_{i\in B_2} b_i\ell^i + \ell^f\sum_{i\not\in B_2} b_i\ell^i \bmod \ell^{2f}-1,$$
or equivalently,
$$\sum_{i\in B_1\setminus B_2} b_i\ell^i \equiv \sum_{i\in B_2\setminus B_1} b_i\ell^i
\bmod \ell^f + 1.$$
Since each sum is less than $2(\ell^f+1)$, they must either be equal or differ by
$\ell^f + 1$, and since $0 \not\in B_1 \cup B_2$, each sum is divisible by $\ell$, 
so in fact equality holds. Since $f-1\in B_1\setminus B_2$, we have
$$b_{f-1}\ell^{f-1} \le \sum_{i\in B_1\setminus B_2} b_i\ell^i =
 \sum_{i\in B_2\setminus B_1} b_i\ell^i \le \sum_{i=2}^{f-1} \ell^i < 2\ell^{f-1},$$
 so $b_{f-1} = 1$. Moreover if $b_{f-2}, b_{f-3},\ldots, b_{s+1}$ are all less
 than $\ell$ for some $s< f-2$, then we must have $b_{f-2}= b_{f-3}= \cdots = b_{s+1}= \ell -1$
 and $f-2,f-3\ldots,s \in B_2 \setminus B_1$, for if either fails, we find that
$$ \sum_{i\in B_2\setminus B_1} b_i\ell^i \le (\ell - 1)\sum_{i=s+1}^{f-2}\ell^i
 + \sum_{i=2}^{s} \ell^i < \ell^{f-1}.$$
 Since $0 \not\in B_2 \setminus B_1$, we conclude that for some $s$ with
 $0 < s < f - 1$, we have $(b_s,b_{s+1},\ldots,b_{f-1}) = (\ell,\ell-1,\ldots,\ell-1,1)$
 and $s,s+1,\ldots, f-2 \in B_2 \setminus B_1$. It follows that
 $$n \equiv \sum_{i \in B_1} b_i \ell^i - \sum_{i\not\in B_1} b_i \ell^i
 \equiv \sum_{i \in B_1, i < s} b_i \ell^i - \sum_{i\not\in B_1, i < s} b_i\ell^i \bmod \ell^f + 1,$$
and this last difference has absolute value at most $\ell(\ell^{f-2}-1)/(\ell - 1)$.
 
Conversely suppose that $\ell^r n \equiv m\bmod \ell^f + 1$ for some $r, m$
with $|m| \le \ell(\ell^{f-2}-1)/(\ell - 1)$. Replacing $r$ by $r+f$ if necessary, 
we may assume $m> 0$ and then
$$ \frac{\ell^s - 1}{\ell - 1} \le m \le \ell\cdot \frac{\ell^s - 1}{\ell - 1}$$
for some $s$ with $0 < s < f - 1$. We can then write
$m = \sum_{i=0}^{s-1} b_i\ell^i $
for some $b_0,b_1,\ldots,b_{s-1}\in \{1,\ldots,\ell\}$.
We can then choose $a \in \Z/(\ell^f - 1)\Z$ so that 
$$\ell^r n \equiv n_{a,\vec{b},B_1} \equiv n_{a,\vec{b},B_2} \bmod \ell^{2f} - 1,$$
where $\vec{b} = (b_0,b_1,\ldots,b_{s-1},\ell,\ell-1,\ldots,\ell - 1, 1)$,
$B_1 = \{0,1,\ldots,f-2\}$ and $B_2 = \{0,1,\ldots,s-1,f-1\}$. 
We conclude that 
$$n \equiv n'_{a',\vec{b}',B_1'} \equiv n'_{a',\vec{b}',B_2'}\bmod \ell^{2f}-1,$$
where $a' \equiv \ell^{-r}a\bmod \ell^f-1$, $b_i' = b_{i+r\bmod f}$ for $i \in \{0,\ldots,f-1\}$
and $B_\nu'$ is such that $J_{B_\nu'} = J_{B_\nu} \circ \frob_\ell^{-r}$ for $\nu = 1,2$.
\end{proof}

\subsection{The reducible case.}

Now suppose $\rho|_D$ is reducible, write
$\rho|_D\sim\begin{pmatrix}\chi_1&{*}\\0&\chi_2\end{pmatrix}$ and let
$c_\rho$ denote the corresponding class in $H^1(K_\gp,\chi_1\chi_2^{-1})$. 
Consider now the set of pairs 
$$W'(\chi_1,\chi_2) = \{\,(V_{\vec{a},\vec{b}},J)\, |\, J\subset S, 
\chi_1|_{I} = \prod_{\tau\in S}\omega_\tau^{a_\tau}\prod_{\tau\in J}\omega_\tau^{b_\tau},
\chi_2|_{I}= \prod_{\tau\in S}\omega_\tau^{a_\tau}\prod_{\tau\not\in J}\omega_\tau^{b_\tau}\,\},$$
with $\pi_1$ and $\pi_2$ denoting the projections
$(V_{\vec{a},\vec{b}},J)\mapsto V_{\vec{a},\vec{b}}$ and
$(V_{\vec{a},\vec{b}},J)\mapsto J$.
Note that interchanging $\chi_1$ and $\chi_2$ replaces $J$ by its complement.
We shall see that the projection map 
$\pi_2: W'(\chi_1,\chi_2)\to \{J\subset S\}$ is typically a bijection, so
that typically $|W'(\chi_1,\chi_2)|=2^f$.
However $W_\gp(\rho)$ will be defined below as
a subset of $\pi_1(W'(\chi_1,\chi_2))$ depending on $c_\rho$. 

We now analyse the set $W'(\chi_1,\chi_2)$ in a manner analogous to the
irreducible case. We write $\chi_\nu = \omega^{n_\nu}$ with $n_\nu \in \Z/(\ell^f -1)\Z$
for $\nu = 1,2$, and we let $n = n_1 - n_2$.
If $a\in \Z/(\ell^f-1)\Z$, $\vec{b} = (b_0,\ldots,b_{f-1})$ with each $b_i\in \{1,\ldots,\ell\}$
and $B\subset \{0,\ldots,f-1\}$, let
$$n_{a,\vec{b},B} = a + \sum_{i\in B}b_i\ell^i \bmod \ell^f-1.$$
Then $W'(\chi_1,\chi_2)$ is in bijection with the set of triples $(a,\vec{b},B)$ as above
with $n_1 \equiv n_{a,\vec{b},B} \bmod \ell^f-1$ and
 $n_2 \equiv n_{a,\vec{b},\ol{B}} \bmod \ell^f-1$ where $\ol{B}$ is the
complement of $B$ in $\{0,\ldots,f-1\}$. Note that for each $B \subset \{0,\ldots,f-1\}$
there is a unique such triple for each solution of
$$n \equiv \sum_{i\in B} b_i\ell^i - \sum_{i\not\in B} b_i\ell^i \bmod \ell^f-1.$$
with $b_0,\ldots,b_{f-1} \in \{1,\ldots,\ell\}$.
But the values of $\sum_{i\in B}b_i\ell^i - \sum_{i\not\in B} b_i\ell^i$
are the $\ell^f$ consecutive integers
from $n_B + 1 - \ell^f $ to $n_B$ where 
$$n_B = \sum_{i\in B}\ell^{i+1} - \sum_{i\not\in B}\ell^i,$$
so there is a unique solution if $n \not\equiv n_B \bmod \ell^f - 1$ and two
solutions if $n \equiv n_B \bmod \ell^f - 1$.
In particular the projection $W'(\chi_1,\chi_2) \to \{J \subset S\}$ is surjective
and $|W'(\chi_1,\chi_2)| = 2^f + |\{\,B\,|\,n \equiv n_B \bmod \ell^f - 1\}|$.

We now show that if $f$ is odd, then the $2^f$ values of $n_B\bmod \ell^f-1$
are distinct, unless $\ell = 2$ or $3$, in which case $n_{\{0,\ldots,f-1\}} \equiv n_\emptyset
\bmod \ell^f - 1 $ and the rest are distinct. First note that
$$n_B \equiv -1 + (\ell + 1)\sum_{i\in B^*} (-1)^i\ell^i \bmod \ell^f - 1,$$
where $B^* = (\{0,2,\ldots,f-1\}\cap B) \cup (\{1,3,\ldots,f-2\}\setminus B)$.
Thus $n_{B_1} \equiv n_{B_2} \bmod \ell^f - 1$ if and only if
$$\sum_{i\in B_1^*} (-1)^i\ell^i \equiv \sum_{i\in B_2^*}(-1)^i\ell^i \bmod (\ell^f - 1)/d,$$
where $d = \gcd(\ell+1,\ell^f - 1)$. If $\ell > 3$, then $d = 2$ and
the two sums differ by
at most $(\ell^f - 1)/(\ell-1) < (\ell^f - 1)/d$, so the above congruence holds
if and only if equality holds, in which case $B_1^* = B_2^*$, so $B_1 = B_2$.
If $\ell = 2$ or $3$, then $d = \ell - 1$, but the two sums differ by $(\ell^f - 1)/(\ell - 1)$
only when one of $B_1$ or $B_2$ is $\{0,\ldots,f-1\}$ and the other is empty.

Now consider the case where $f$ is even and $\ell> 3$.
We then have
$$n_B \equiv (\ell + 1)\sum_{i\in B^*} (-1)^i\ell^i \bmod \ell^f - 1,$$
where $B^* = (\{0,2,\ldots,f-2\}\cap B) \cup (\{1,3,\ldots,f-1\}\setminus B)$.
Thus $n_{B_1} \equiv n_{B_2} \bmod \ell^f - 1$ if and only if
$$\sum_{i\in B_1^*} (-1)^i\ell^i \equiv \sum_{i\in B_2^*}(-1)^i\ell^i \bmod (\ell^f - 1)/(\ell+1).$$
But these two sums differ by
at most $(\ell^f - 1)/(\ell-1)$, which is less than $2 (\ell^f - 1)/(\ell + 1)$.
So if the above congruence holds
then either equality holds, in which case $B_1= B_2$, or 
$$\sum_{i\in B_2^*} (-1)^i\ell^i = \sum_{i\in B_1^*}(-1)^i\ell^i + \sum_{i=0}^{f-1}(-1)^i\ell^i,$$
exchanging $B_1$ and $B_2$ if necessary. Solving $\bmod \ell^r$ inductively on $r$,
we see that the only possibility is that $B_1^* = \emptyset$ and $B_2^* = \{0,1,\ldots,f-1\}$,
in which case $n_{B_\nu} \equiv 0 \bmod \ell^f-1$.

If $f$ is even and $\ell = 3$, then the situation is the same, except that we have
$(\ell^f - 1)/(\ell-1) = 2 (\ell^f - 1)/(\ell + 1)$, so in addition to the possibilities that
arose for $\ell > 3$, we have 
$$n_{\{0,\ldots,f-1\}} \equiv n_\emptyset \equiv (\ell^f-1)/2\bmod \ell^f - 1$$
as for odd $f$.

Finally suppose that $f$ is even and $\ell = 2$. In this case we have 
$$n_B \equiv 3\sum_{i\in B^*} (-1)^i 2^i \bmod 2^f -1,$$
where $B^* = (\{0,2,\ldots,f-2\}\cap B) \cup (\{1,3,\ldots,f-1\}\setminus B)$.
In particular $n_B \equiv 0 \bmod 3$. Moreover we have
$|n_{B_1} - n_{B_2}| \le 3(2^f - 1)$ with equality possible only when
one of $B_1$ or $B_2$ is $\{0,\ldots,f-1\}$ and the other is empty,
in which case $n_{B_\nu} \equiv 0 \bmod 2^f - 1$. Thus each 
non-zero congruence class
$\bmod\ 2^f - 1$ arises as $n_B$ for at most $3$ values of $B$, while $0$
arises for at most $4$. It follows
that each of the $(2^f-4)/3$ non-zero multiples of $3\bmod 2^f -1$ arises as $n_B$
for exactly three values of $B$, while $0$ arises for exactly four values. 

We have thus proved the following propositions:

\begin{proposition} Suppose that $\ell> 3$. If $f$ is odd, then the
congruence classes $\bmod\ \ell^f - 1$ of the form
$$ -1 + (\ell + 1)\sum_{i\in B^*} (-1)^i\ell^i$$
are distinct and non-zero as $B^*$ runs through all subsets of $\{0,1,\ldots,f-1\}$.
 If $f$ is even, then the 
congruence classes $\bmod\ \ell^f - 1$ of the form
$$(\ell + 1)\sum_{i\in B^*} (-1)^i\ell^i$$
are distinct and non-zero as $B^*$ runs through all non-empty proper subsets of 
$\{0,1,\ldots,f-1\}$. Letting $A$ denote the set of such classes in each case, we have
$$|W'(\chi_1,\chi_2)| = \left\{\begin{array}{ll}
2^f + 2, & \mbox{if $n\equiv0$ mod~$\ell^f-1$ and $f$ is even,}\\
2^f + 1, & \mbox{if $n\in A$,}\\
2^f, & \mbox{otherwise.}
\end{array}\right.$$
\end{proposition}

\begin{proposition} Suppose that $\ell = 3$. 
If $f$ is odd, then the
congruence classes $\bmod\ 3^f - 1$ of the form
$$ -1 +4\sum_{i\in B^*} (-1)^i3^i$$
are distinct and non-zero $\bmod\ (3^f - 1)/2$ as $B^*$ runs 
through all the subsets of $\{0,1,\ldots,f-1\}$ other than
$\{0,2,\ldots,f-1\}$ and $\{1,3,\ldots,f-2\}$.
If $f$ is even, then the 
congruence classes $\bmod\ 3^f - 1$ of the form
$$4\sum_{i\in B^*} (-1)^i3^i$$
are distinct and non-zero $\bmod\ (3^f-1)/2$ as $B^*$ runs 
through all non-empty proper subsets of 
$\{0,1,\ldots,f-1\}$ other than 
$\{0,2,\ldots,f-2\}$ and $\{1,3,\ldots,f-1\}$.
Letting $A$ denote the set of such classes in each case, we have
$$|W'(\chi_1,\chi_2)| = \left\{\begin{array}{ll}
2^f + 2, & \mbox{if $n\equiv 0$~mod~$\ell^f-1$ and $f$ is even, or $n \equiv (\ell^f-1)/2$}\\
2^f + 1, & \mbox{if $n\in A$,}\\
2^f, & \mbox{otherwise.}
\end{array}\right.$$
\end{proposition}

\begin{proposition} Suppose that $\ell = 2$. Then (all congruences on $n$ below being modulo $2^f-1$) we have
$$|W'(\chi_1,\chi_2)| = \left\{\begin{array}{ll}
2^f+4, & \mbox{if $n\equiv0$ and $f$ is even,}\\
2^f+3 & \mbox{if $n\not\equiv 0$, $3\mid n$ and $f$ is even,}\\
2^f+2, & \mbox{if $n\equiv0$ and $f$ is odd,}\\
2^f+1, & \mbox{if $n\not\equiv0$ and $f$ is odd,}\\
2^f,& \mbox{if $3\nmid n$ and $f$ is even.}\\
\end{array}\right.$$
\end{proposition}
Note that ``$3\mid n$'' is well-defined when $f$ is even (as then $n$
is an integer modulo $4^{f/2}-1$ which is zero mod~3).
As in the irreducible case, multiple $B$ can occur with the same $(a,\vec{b})$.

\begin{proposition} The projection map from $W'(\chi_1,\chi_2)$ onto its first component fails to be injective
if and only if $\ell^r n \equiv m\bmod \ell^f - 1$ for some integers $r, m$
with $|m| \le \max\{0,\ell(\ell^{f-2}-1)/(\ell - 1)\}$.
\end{proposition}
\begin{proof} The statement that the projection map from $W'(\chi_1,\chi_2)$
to its first component is not injective is equivalent to the statement
that for some $a$, $\vec{b}$ and $B_1\neq B_2$, we have 
$$n_1 \equiv n_{a,\vec{b},B_1} \equiv n_{a,\vec{b},B_2}\bmod\ell^f -1$$
and
$$n_2 \equiv n_{a,\vec{b},\overline{B}_1} \equiv n_{a,\vec{b},\overline{B}_2}\bmod\ell^f -1,$$
where $\overline{B}$ denotes the complement of $B$ in $\{0,1,2,\ldots,f-1\}$.

We first deal with the special case $n\equiv 0 \bmod \ell^f- 1$. In this case the map is never
injective; take $\vec{b}=(\ell-1,\ell-1,\ldots,\ell-1)$, 
$B_1=\{0,1,\ldots,f-1\}$ and $B_2=\emptyset$ (with the appropriate value of~$a$).

So let us now assume that $n\not\equiv 0 \bmod \ell^f -1$. Suppose first
that the projection map is not injective. Because $n\not\equiv0$
we check that $B_2$ cannot be the complement of $B_1$ (note that this finishes
the proof in the case $f=1$). We can then assume $f-1 \in B_1 \setminus B_2$
and $0 \not\in (B_1\setminus B_2) \cup (B_2 \setminus B_1)$ after exchanging $B_1$ and $B_2$ and multiplying
$n$ by a power of $\ell$ if necessary,
and complete the argument as in the irreducible case.
\end{proof}

For each pair $\alpha = (V_{\vec{a},\vec{b}},J)\in W'(\chi_1,\chi_2)$ we shall define below a
subspace $L_\alpha \subset H^1(K_\gp,\ol{\F}_\ell(\chi_1\chi_2^{-1}))$ of dimension
$|J|+\delta$, where $\delta = 0$ except in certain cases where $\chi_1\chi_2^{-1}$
is trivial or cyclotomic. We then define $W_\gp(\rho)$ by the following rule: 
$$V_{\vec{a},\vec{b}}\in W_\gp(\rho) \mbox{\ if and only if\ }c_\rho \in L_\alpha
\mbox{\ for some\ }\alpha = (V_{\vec{a},\vec{b}},J)\in W'(\chi_1,\chi_2).$$
Note in particular that $W_\gp(\rho)\subset\pi_1(W'(\chi_1,\chi_2))$, with
equality if $c_\rho=0$, that is, if $\rho|D$ is split.

Before defining the subspace $L_\alpha$, we recall some facts
about crystalline representations. 
Recall that a character $\psi:D \to \ol{\Q}_\ell^\times$ is crystalline if and only the
filtered $\phi$-module
$D_\crys(\psi) = (B_\crys \otimes_{\Q_\ell} \ol{\Q}_\ell(\psi))^D$ is free of
rank one over $K_\gp \otimes_{\Q_\ell} \ol{\Q}_\ell$.

For each $\tau \in S$, let $e_\tau: K_\gp \otimes_{\Q_\ell} \ol{\Q}_\ell \to \ol{\Q}_\ell$ denote
the projection defined by $a\otimes b \mapsto \tilde{\tau}(a)b$ where 
$\tilde{\tau}$ is the embedding $K_\gp \to \ol{\Q}_\ell$ reducing to $\tau$,
and let $e_\tau D_\crys(\psi)$ denote the filtered $\ol{\Q}_\ell$-vector space 
$D_\crys(\psi)\otimes_{K_\gp\otimes\ol{\Q}_\ell,e_\tau}\ol{\Q}_\ell$.
\begin{lemma} \label{lem:inertia}
Suppose that $\psi$ is a crystalline character and for each $\tau\in S$,
$m_\tau$ is the integer such that $\gr^{-m_\tau} e_\tau D_\crys(\psi) \neq 0$.
Then $\ol{\psi}|_I = \prod_{\tau\in S} \omega_\tau^{m_\tau}$.
\end{lemma}
\begin{proof} Crystalline characters satisfying the first condition correspond to 
(weakly) admissible filtered $\phi$-modules with the specified filtration. These exist, 
and any two such differ by an unramified twist. Taking tensor products, the lemma 
reduces to the case where $m_\tau = 1$ if $\tau=\tau_0$ and
$m_\tau=0$ otherwise. The result in this case follows for example from
Theorems~5.3 and 8.4 of \cite{fon_laf}.
\end{proof}

Recall that $\alpha = (V_{\vec{a},\vec{b}},J)\in W'(\chi_1,\chi_2)$ if and only if
$$\chi_1|_{I}= 
\prod_{\tau\in S}\omega_\tau^{a_\tau}\prod_{\tau\in J}\omega_\tau^{b_\tau},\qquad
\chi_2|_{I}
= \prod_{\tau\in S}\omega_\tau^{a_\tau}\prod_{\tau\not\in J}\omega_\tau^{b_\tau}.$$
\begin{lemma}\label{lem:lift}
Suppose that $\alpha = (V_{\vec{a},\vec{b}},J) \in W'(\chi_1,\chi_2)$. 
Let $m_{\tau,\alpha} = b_\tau$ if $\tau\in J$ and
 $m_{\tau,\alpha}=-b_\tau$ if $\tau \not\in J$.
Then there is a unique lift $\chi_{\alpha}$ of $\chi_1\chi_2^{-1}$ with the following 
properties:
\begin{itemize} 
\item $\chi_{\alpha}$ is crystalline with $\gr^{- m_\tau} e_\tau D_\crys(\chi_{\alpha}) \neq 0$
for each $\tau \in S$;
\item if $g \in D^\ab$ corresponds via local class field theory to $\ell$, then 
$\chi_{\alpha}(g)$ is the Teichm\"uller lift of $\chi_1\chi_2^{-1}(g)$. 
\end{itemize}
\end{lemma}
\begin{proof} Let $\psi$ be a character satisfying the first condition. The preceding lemma
shows that the reduction of $\psi$ is an unramified twist of $\chi_1\chi_2^{-1}$. 
Let $\chi_{\alpha}
= \delta \psi$, where $\delta$ is the unramified
character with $\delta(g) = \tilde{\chi}_1\tilde{\chi}_2^{-1}(g)\psi^{-1}(g)$ (and the tildes denote Teichm\"uller lifts).
\end{proof}

Recall that if $\psi: D \to \ol{\Q}_\ell^\times$ is a crystalline representation,
then Bloch and Kato define a subspace $H^1_f(K_\gp,\ol{\Q}_\ell(\psi))$ corresponding to
those extensions of topological $\ol{\Q}_\ell D$-modules
$$0 \to \ol{\Q}_\ell(\psi) \to E \to \ol{\Q}_\ell \to 0$$
which are crystalline. By Corollary~3.8.4 of~\cite{bloch_kato} for example, 
we have 
$$\dim H^1_f(K_\gp,\ol{\Q}_\ell(\psi)) = \dim H^0(K_\gp,\ol{\Q}_\ell(\psi)) + \dim D_\crys(\psi)
 - \dim \fil^0 D_\crys(\psi)$$
 where the dimensions are over $\ol{\Q}_\ell$.
Applying this for $\psi = \chi_\alpha$ for
$\alpha = (V_{\vec{a},\vec{b}},J)\in W'(\chi_1,\chi_2)$, we see that
$\dim H^1_f(K_\gp,\ol{\Q}_\ell(\chi_\alpha)) = |J|$.
We then define $H^1_f(K_\gp,\ol{\Z}_\ell(\chi_\alpha))$ as the preimage
of $H^1_f(K_\gp,\ol{\Q}_\ell(\chi_\alpha))$ under the natural map
$$H^1(K_\gp,\ol{\Z}_\ell(\chi_\alpha))\to H^1(K_\gp,\ol{\Q}_\ell(\chi_\alpha))$$
and $L_\alpha'$ as the image of $H^1_f(K_\gp,\ol{\Z}_\ell(\chi_\alpha))$
under the natural map
$$H^1(K_\gp,\ol{\Z}_\ell(\chi_\alpha))\to H^1(K_\gp,\ol{\F}_\ell(\chi_1\chi_2^{-1})).$$
We then let $L_\alpha = L_\alpha'$ except in two cases:
\begin{itemize}
\item if $\chi_1\chi_2^{-1}$ is the cyclotomic character,
 $\vec{b} = (\ell,\ldots,\ell)$ and $J = S$, then we let 
 $L_\alpha = H^1(K_\gp,\ol{\F}_\ell(\chi_1\chi_2^{-1}))$;
\item if $\chi_1\chi_2^{-1}$ is the trivial character and $J \neq S$, then we let 
 $L_\alpha = L_\alpha' + L_\ur$, where $L_\ur$ is the one-dimensional space
of unramified classes in $H^1(K_\gp,\ol{\F}_\ell)$.
 \end{itemize}
 
\begin{remark} \label{rmk:indep} Recall that in defining $\chi_\alpha$ we chose a lift (namely the
Teichm\"uller lift) of $\chi_1\chi_2^{-1}(g)$.  If $b_\tau < \ell$ for all $\tau\in S$,
then one can show using Fontaine-Laffaille theory \cite{fon_laf} that the space
$L_{\alpha}'$ is independent of the choice of lift.  We expect this to be true
even if some $b_\tau=\ell$, except in the two cases where we accordingly modified the
definition of $L_\alpha$.  This independence of lift is proved for $f=2$ by Chang and
one of the authors in \cite{cd_new}, Remark~7.13.
\end{remark}

\begin{remark} \label{rmk:peuramifie}  If $\chi_1\chi_2^{-1}$ is the cyclotomic character,
$\vec{b}=(1,\ldots,1)$ and $J=S$, then $\chi_\alpha$ is the $\ell$-adic cyclotomic character.
In that case we have the isomorphism
$H^1(K_\gp,\ol{\Q}_\ell(\chi_\alpha)) \cong (K_\gp^\times)\hat{\ }\otimes_{\Z_\ell}\ol{\Q}_\ell$
by Kummer theory (where $\hat{\cdot}$ denotes $\ell$-adic completion), and one
knows that
$H^1_f(K_\gp,\ol{\Q}_\ell(\chi_\alpha))\cong (\CO_{K_\gp}^\times)\hat{\ }\otimes_{\Z_\ell}\ol{\Q}_\ell$
(these being the extensions arising from $\ell$-divisible groups).  It follows that
$L_\alpha$ corresponds to $\CO_{K_\gp}^\times\otimes\ol{\F}_\ell$ under the isomorphism
$H^1(K_\gp,\ol{\F}_\ell(\chi_1\chi_2^{-1})) \cong K_\gp^\times\otimes\ol{\F}_\ell$, hence
consists of the classes which are {\em peu ramifi\'ees} in the sense of Serre~\cite{serre:duke}.
\end{remark}
 
 \begin{lemma}\label{lem:dim} If $\alpha = (V_{\vec{a},\vec{b}},J)\in W'(\chi_1,\chi_2)$, then
 $\dim L_\alpha = |J|$ except in the following cases:
 \begin{enumerate}
 \item if $\chi_1\chi_2^{-1}$ is cyclotomic, $\vec{b} = (\ell,\ldots,\ell)$, $J=S$
 and $\ell > 2$, then $\dim L_\alpha = |J|+1$;
 \item if $\chi_1\chi_2^{-1}$ is trivial, then $\dim L_\alpha = |J| + 1$ 
 unless either $L_\ur \not\subset L_\alpha'$
 or $\vec{b} = (\ell,\ldots,\ell)$, in which case $\dim L_\alpha = |J|+2$.
 \end{enumerate}
 \end{lemma}
 \begin{proof} Note first that $H^1_f(K_\gp,\ol{\Z}_\ell(\chi_\alpha))$
contains $H^1(K_\gp,\ol{\Z}_\ell(\chi_\alpha))_\tor$ and that
the quotient is free of rank $|J| =\dim H^1_f(K_\gp,\ol{\Q}_\ell(\chi_\alpha))$.
Therefore the natural map
$$ H^1_f(K_\gp,\ol{\Z}_\ell(\chi_\alpha))\otimes_{\ol{\Z}_\ell}\ol{\F}_\ell \to
H^1(K_\gp,\ol{\Z}_\ell(\chi_\alpha))\otimes_{\ol{\Z}_\ell}\ol{\F}_\ell \to
 H^1(K_\gp,\ol{\F}_\ell(\chi_1\chi_2^{-1}))$$
is injective and its image $L_\alpha'$ has dimension $|J|$, unless
$\chi_1=\chi_2$, in which case the dimension is $|J|+1$.
If $\chi_1\neq \chi_2$, then the
lemma is now immediate from the definition of $L_\alpha$ and,
in the cyclotomic $J =S$ case, the local Euler
characteristic formula.  If $\chi_1 = \chi_2$, then one also
needs to note the following:
\begin{itemize}
\item If $J=S$ and $\vec{b}\neq(\ell,\ldots,\ell)$, then $\vec{b}=(\ell-1,\ldots,\ell-1)$
 and $L_\ur \subset L_\alpha'$, for dimension reasons if $\ell > 2$ and by Remark~\ref{rmk:peuramifie}
 if $\ell = 2$.
\item If $\vec{b}=(\ell,\ldots,\ell)$, then $\ell=2$ and $J$ is either $S$ or $\emptyset$.
  In the latter case, $L_\alpha'$ is spanned by the homomorphism $G_{K_\gp} \to \ol{\F}_2$ defined
  by the reduction of $(\omega^2-1)/8$, where $\omega$ is the $2$-adic cyclotomic
  character.  So in this case $L_\ur \not\subset L_\alpha'$.
\end{itemize}
\end{proof}

If $\rho|_D \sim\begin{pmatrix}\chi_1&{*}\\0&\chi_2\end{pmatrix}$ and
$c_\rho$ is the corresponding class in $H^1(K_\gp,\ol{\F}_\ell(\chi_1\chi_2^{-1}))$,
we now define 
\begin{equation}\label{eqn:red}
W_\gp(\rho)= \{\, V\, |\, \exists J\mbox{\ with\ } c_\rho \in L_\alpha \mbox{\ for $\alpha = (V,J) \in W'(\chi_1,\chi_2)$}\,\}. \end{equation}
Note that if $\rho|_D\sim \chi_1\oplus\chi_2$, or equivalently $c_\rho = 0$,
then $W_\gp(\rho) = \pi_1(W'(\chi_1,\chi_2))$ is independent of the choice of ordering of 
$\chi_1$ and $\chi_2$. Note also in this case that $|W_\gp(\rho)|$
has size approximately $2^f$, as
in the irreducible case (\ref{eqn:irred}).  Moreover if $\rho$ is reducible,
one knows by \cite{breuil_new}, Prop~A.3 or \cite{cd_new}, Thm.~7.8 that
$|W_\gp(\rho)| = 2^d$ for some $d\in\{0,\ldots,f\}$ depending on $c_\rho$, 
provided $\chi_1\chi_2^{-1}$
is {\em generic} in the sense that its restriction to $I_{K_\gp}$ is of the form
$\prod_{\tau\in S}\omega_\tau^{b_\tau}$ with $b_\tau \in \{1,\ldots,\ell-2\}$
for all $\tau\in S$ and $\vec{b} \not\in \{(1,\ldots,1),(\ell-2,\ldots,\ell-2)\}$.

Finally we remark that it is shown in \cite{fred:durham} that in the
cases where $\rho|_D$ is semisimple, the
set $W_\gp(\rho)$ is related to the set of Jordan-H\"older constituents of the reduction
of a corresponding irreducible characteristic zero representation of $\GL_2(k)$.

\subsection{Basic properties of the definition.}
The set $W_\gp(\rho)$ was defined in terms of the restriction of $\rho$
to $G_{K_\gp}$. We now check that it is in fact non-empty and
depends only on the restriction to inertia.
\begin{proposition}\label{prop:inertia}
If $\rho:G_K \to \GL_2(\ol{\F}_\ell)$ is
continuous, irreducible and totally odd, then
$W_\gp(\rho)$ is non-empty and
depends only on $\rho|_{I_{K_\gp}}$.
\end{proposition}

\begin{proof} 
We first prove that $W_\gp(\rho)\neq \emptyset.$
If $\rho|_{G_{K_\gp}}$ is irreducible,
then it is induced from a character $\xi$, and
Propositions~\ref{prop:irred} and~\ref{prop:irred2} show that 
$W'(\xi)$ is non-empty, and hence so is $W_\gp(\rho)$
(note that if $\ell^f = 2$, then $n$ is not divisible by $3$).
If $\rho|_{G_{K_\gp}}$ is reducible, then it is of the
form $\begin{pmatrix}\chi_1&{*}\\0&\chi_2\end{pmatrix}$, and we showed
that the projection map $W'(\chi_1,\chi_2)\to \{J\subset S\}$
is surjective. In particular, there is an element
$\alpha = (V_{\vec{a},\vec{b}},S) \in W'(\chi_1,\chi_2)$. 
Moreover if $\chi_1\chi_2^{-1}$ is cyclotomic, we may
choose $\vec{b} = (\ell,\ldots,\ell)$, so that in all
cases $L_\alpha = H^1(K_\gp,\ol{\F}_\ell(\chi_1\chi_2^{-1}))$
by Lemma~\ref{lem:dim} and the local Euler characteristic
formula. It follows that $c_\rho \in L_\alpha$ and $V_{\vec{a},\vec{b}}
\in W_\gp(\rho)$.

For the dependence only on inertia, 
first note that the irreducibility of $\rho|_{G_{K_\gp}}$
is determined by $\rho|_{I_{K_\gp}}$ (see for example section~2.4
of~\cite{edixhoven}). In the case that 
$\rho|_{G_{K_\gp}}$ is irreducible, $W_\gp(\rho)$ is
determined by $W'(\xi)$, which depends only on $\xi|_{I_{K_\gp}}$,
which in turn depends only on $\rho|_{I_{K_\gp}}$.

Suppose now that $\rho|_{G_{K_\gp}}$ is reducible and
$\rho':G_K \to \GL_2(\ol{\F}_\ell)$ is such that 
$\rho|_{I_{K_\gp}}\sim \rho'|_{I_{K_\gp}}$. Changing
bases, we may assume $\rho|_{G_{K_\gp}} =
 \begin{pmatrix}\chi_1&{*}\\0&\chi_2\end{pmatrix}$
 and $\rho|_{I_{K_\gp}}= \rho'|_{I_{K_\gp}}$.
Note that the function $G_{K_\gp}/I_{K_\gp}\to \GL_2(\ol{\F}_\ell)$
 defined by $g \mapsto \rho'(g)\rho(g)^{-1}$ takes values in
 $Z(\rho(I_{K_\gp}))$. We divide the proof into cases
 according to the possible centralisers.
 
 Suppose first that $\chi_1\chi_2^{-1}$ is ramified and $c_\rho $ has 
 non-trivial restriction to $I_{K_\gp}$. In this case
$\rho|_{I_{K_\gp}}$ is indecomposable and has centraliser
 consisting only of the scalar matrices. It follows that $\rho' 
 =\psi\rho$ for some unramified character $\psi:G_{K_\gp} \to
 \ol{\F}_\ell^\times$. Since $W'(\chi_1,\chi_2)$ depends only
 on the restriction of $\chi_1$ and $\chi_2$ to $I_{K_\gp}$,
 we have $W'(\psi\chi_1,\psi\chi_2)=W'(\chi_1,\chi_2)$. Moreover the subspaces
 $L_\alpha$ of $H^1(K_\gp,\ol{\F}_\ell(\chi_1\chi_2^{-1}))$
 do not change if $\rho$ is replaced by an unramified twist,
 nor does the class $c_\rho$. It
 follows that $W_\gp(\rho') = W_\gp(\rho)$.
 
 Next suppose that $\chi_1\chi_2^{-1}$ is ramified and 
 $c_\rho$ has trivial restriction to $I_{K_\gp}$. Then in fact
 $c_\rho = 0$, so we may assume $\rho|_{G_{K_\gp}} = 
 \begin{pmatrix}\chi_1&{0}\\0&\chi_2\end{pmatrix}$. In this case the
 centraliser of $\rho(I_{K_\gp})$ consists of the diagonal matrices.
It follows that $\rho'|_{G_{K_\gp}} = 
 \begin{pmatrix}\psi_1\chi_1&{0}\\0&\psi_2\chi_2\end{pmatrix}$
 for some unramified characters $\psi_1$ and $\psi_2$.
Hence $c_{\rho'} = 0$ and $W'(\psi_1\chi_1,\psi_2\chi_2)
= W'(\chi_1,\chi_2)$, so $W_\gp(\rho') = W_\gp(\rho)$.

Next suppose that $\chi_1\chi_2^{-1}$ is unramified and $c_\rho$
has non-trivial restriction to $I_{K_\gp}$. In this case we have
$$Z(\rho(I_{K_\gp})) = \left\{\, \left.
 \begin{pmatrix}x&{xy}\\0&x \end{pmatrix}\,\right|\,x\in\ol{\F}_\ell^\times,
 y\in\ol{\F}_\ell\,\right\},$$
so if $g \in G_{K_\gp}$, then 
$$\rho'(g) = 
 \begin{pmatrix}1&\mu(g)\\0&1\end{pmatrix}\psi(g)\rho(g)$$
 for some unramified character $\psi$ and cocycle 
 $\mu:G_{K_\gp}/I_{K_\gp} \to \ol{\F}_\ell(\chi_1\chi_2^{-1})$.
In particular, $\rho'\sim \begin{pmatrix}\chi_1'&{*}\\0&\chi_2'\end{pmatrix}$
with $\chi'_1(\chi_2')^{-1} = \chi_1\chi_2^{-1}$.  Moreover
 if $\chi_1 \neq \chi_2$, then $c_{\rho'} = c_\rho$
 in $H^1(K_\gp,\ol{\F}_\ell(\chi_1\chi_2^{-1}))$, and if $\chi_1 = \chi_2$,
 then $c_{\rho'} - c_\rho \in L_\ur$. Since the spaces $L_\alpha$ are
 the same for $\rho$ and $\rho'$ and contain $L_\ur$ if $\chi_1 = \chi_2$,
 we conclude that $W_\gp(\rho') = W_\gp(\rho)$.
 
 Finally suppose that $\chi_1\chi_2^{-1}$ is unramified and
 $c_\rho$ has trivial restriction to $I_{K_\gp}$, so
 $$\rho'(g) = \rho(g) = 
 \begin{pmatrix}\chi_1(g)&0\\0&\chi_1(g)\end{pmatrix}$$ 
 for $g \in I_{K_\gp}$. Note that
 $c_\rho = 0$ unless $\chi_1 = \chi_2$, in which case
 $c_\rho \in L_\ur$, and similarly for $c_{\rho'}$.
 It follows that $W_\gp(\rho') = \pi_1(W'(\chi_1',\chi_2'))
 = \pi_1(W'(\chi_1,\chi_2)) = W_\gp(\rho)$. 
\end{proof}

We are now ready to state the weight conjecture.
Recall that $W(\rho)$ 
is defined as the set of representations of $\prod_{\gp|\ell}\GL_2(k_\gp)$
of the form $\otimes_{\ol{\F}_\ell} V_\gp$
with each $V_\gp \in W_{\gp}(\rho)$.
By the preceding proposition $W(\rho)$ is non-empty and
depends only on the restrictions of $\rho$ to inertia
groups at primes over $\ell$.
\begin{conjecture} \label{conj:weight}
If $\rho:G_K \to \GL_2(\ol{\F}_\ell)$ is modular, then
$$W(\rho) = \{\,V\,|\,\rho \mbox{ is modular of weight $V$}\}.$$
\end{conjecture}

We now check compatibility of the conjectural weight set
with twists and determinants.
\begin{proposition} \label{prop:chars}
Suppose that $\rho:G_K \to \GL_2(\ol{\F}_\ell)$ is
continuous, irreducible and totally odd.
\begin{enumerate}
\item 
Let $\chi:G_K \to \ol{\F}_\ell^\times$ be such that
$\chi|_{I_{K_\gp}}= \prod_{\tau\in S_\gp}\omega_\tau^{c_\tau}$
for each $\gp|\ell$. Then $V \in W(\rho)$ if and only 
$V\otimes V_\chi\in W(\chi\rho)$, where
$$V_\chi = \otimes_{\gp|\ell} \otimes_{\tau\in S_\gp} \det{}^{c_\tau}
\otimes_\tau\ol{\F}_\ell$$
is an $\overline{\F}_\ell$-valued character of $G$.
\item 
If $V\in W(\rho)$ and $V$ has central character $\otimes_{\gp|\ell}
\prod_{\tau\in S_\gp} \tau^{c_\tau}$,
then $\det\rho|_{I_{K_\gp}}= \prod_{\tau\in S_\gp}\omega_\tau^{c_\tau+1}$
for each $\gp|\ell$.
\end{enumerate}
\end{proposition}
\begin{proof} To prove the first assertion it suffices to show that
$V \in W_\gp(\rho)$ if and only if $V\otimes V_{\chi_\gp}\in W_\gp(\chi\rho)$
where $\chi_\gp = \chi|_{G_{K_\gp}}$ and $V_{\chi _\gp} = 
\otimes_{\tau\in S_\gp} \det{}^{c_\tau}
\otimes_\tau\ol{\F}_\ell$. 
If $\rho|_{G_{K_\gp}}\sim\Ind_{D'}^D\xi$ is irreducible, 
then $(V,J) \in W'(\xi)$ if and only $(V\otimes V_{\chi_\gp},J)
\in W'(\xi\chi_\gp)$, yielding the assertion in this case.
If $\rho|_{G_{K_\gp}}\sim
 \begin{pmatrix}\chi_1&{*}\\0&\chi_2\end{pmatrix}$ is reducible, 
then $\alpha = (V,J) \in W'(\chi_1,\chi_2)$ if and only 
$\alpha' = (V\otimes V_{\chi_\gp},J)\in W'(\chi_1\chi_\gp,\chi_2\chi_\gp)$.
Moreover, since $V_{\vec{a},\vec{b}} \otimes V_{\chi_\gp}$ is of the form
$V_{\vec{a'},\vec{b}}$, we see that $\chi_\alpha = \chi_{\alpha'}$,
 so $L_\alpha = L_{\alpha'} \subset 
 H^1(K_\gp,\ol{\F}_\ell(\chi_1\chi_2^{-1}))$. Since also
 $c_\rho = c_{\chi\rho}$, we get the assertion in this case
 as well.
 
To prove the second assertion, we can again work locally at
primes $\gp|\ell$.
Writing $V = \otimes_{\gp|\ell} V_\gp$, we have $V_\gp \in W_\gp(\rho)$
for each $\gp | \ell$. If $V_\gp = V_{\vec{a},\vec{b}}$, this
gives $\det\rho|_{I_{K_\gp}} =
\prod_{\tau\in S_\gp} \omega_\tau^{2a_\tau + b_\tau}$. Since
 $V_{\vec{a},\vec{b}}$ has central character $
\prod_{\tau\in S_\gp} \tau^{2a_\tau + b_\tau - 1}$, 
we have
$$\sum_{i=0}^{f_\gp - 1} c_{\tau_i} \ell^i \equiv 
\sum_{i=0}^{f_\gp - 1} (2a_{\tau_i} + b_{\tau_i} -1)\ell^i
\bmod (\ell^{f_\gp} - 1).$$
Adding $\sum_{i=0}^{f_\gp-1} \ell^i$ to each side
of the congruence, we deduce that 
$\prod_{\tau\in S_\gp}\omega_\tau^{c_\tau+1}
= \prod_{\tau\in S_\gp}\omega_\tau^{2a_\tau+b_\tau}.$
\end{proof}

Combining the first part of the proposition with
Corollary~\ref{gl1}(b), we deduce the following:
\begin{corollary} \label{cor:twist}
Suppose that $\rho:G_K \to \GL_2(\ol{\F}_\ell)$ is
continuous, irreducible and totally odd and
$\chi:G_K \to \ol{\F}_\ell^\times$ is a character.
Then Conjecture~\ref{conj:weight} holds for $\rho$
if and only if it holds for $\chi\rho$.
\end{corollary}

In the case $K=\Q$, Conjecture~\ref{conj:weight} follows from known results on Serre's Conjecture
(see~\cite{khare:pp}, Thm.~5 and \cite{ash1}, \S2 for similar statements along these lines).
\begin{theorem} \label{thm:Qconj}
Conjecture~\ref{conj:weight} holds if $K = \Q$.
\end{theorem}
\begin{proof} Replacing $\rho$ by a twist, we can assume $\rho|_{I_\ell}$ has the form 
$ \begin{pmatrix}\omega_2^b&{0}\\0&\omega_2^{\ell b}\end{pmatrix}$
or $ \begin{pmatrix}\omega^b&{*}\\0&1\end{pmatrix}$
for some $b$ with $1\le b \le \ell-1$. In the second case we write
$\rho|_{G_{\Q_\ell}}\sim \begin{pmatrix}\chi_1&{*}\\0&\chi_2\end{pmatrix}$.
We shall also use $\omega$ to denote the mod $\ell$ cyclotomic character on $G_\Q$
or $G_{\Q_\ell}$.

In the first case we find that $W(\rho) = \{V_{0,b},V_{b-1,\ell+1-b}\}$.
In the second case, we have the following possibilities:
$$W'(\chi_1,\chi_2) = \left\{\begin{array}{ll}
\{(V_{0,b},S),(V_{b,\ell-1-b},\emptyset)\},&
\mbox{if $1 < b < \ell - 2$,}\\
\{(V_{0,\ell-1},S),(V_{0,\ell-1},\emptyset)\},&
\mbox{if $b = \ell - 1$ and $\ell > 2$,}\\
\{(V_{0,\ell},S),(V_{0,1},S),(V_{1,\ell-2},\emptyset)\},&
\mbox{if $b = 1$ and $\ell > 3$,}\\
\{(V_{0,\ell-2},S),(V_{\ell-2,\ell},\emptyset),(V_{\ell-2,1},\emptyset)\},&
\mbox{if $b = \ell-2$ and $\ell > 3$,}\\
\{(V_{0,\ell},S),(V_{0,1},S),(V_{1,\ell},\emptyset),(V_{1,1},\emptyset)\},&
\mbox{if $b = 1$ and $\ell \le 3$.}\end{array}\right.$$
Moreover, dimension considerations show that
$L_{(V,J)} = H^1(G_{\Q_\ell},\ol{\F}(\chi_1\chi_2^{-1}))$ whenever
$J=S$, unless $\chi_1\chi_2^{-1}=\omega$ and $V=V_{(0,1)}$, in which
case $L_{(V,J)}$ consists of the
peu ramifi\'ees classes (see Remark~\ref{rmk:peuramifie}).
Considering dimensions in this case then gives that
$L_{(V,J)}$ consists precisely of the classes
which are peu ramifi\'ees. Note also that
$L_{(V,J)} = 0$ whenever $J = \emptyset$ unless
$\chi_1\chi_2^{-1}$ is trivial, in which case all we need to know
is that if $\ell=2$, then $L_{(V_{0,1},\emptyset)}$ is the set of
peu ramifi\'ees classes by a direct calculation.
This gives
$$W(\rho) = \left\{\begin{array}{ll}
\{V_{0,b}\},&
\mbox{if $1 < b < \ell - 1$ and $\rho|_{G_{\Q_\ell}}$ is non-split,}\\
\{V_{0,b},V_{b,\ell-1-b}\},&
\mbox{if $1 < b < \ell - 2$ and $\rho|_{G_{\Q_\ell}}$ is split,}\\
\{V_{0,\ell-2}, V_{\ell-2,\ell}, V_{\ell-2,1}\},&
\mbox{if $b = \ell - 2$, $\ell > 3$ and $\rho|_{G_{\Q_\ell}}$ is split,}\\
\{V_{0,\ell-1}\},&
\mbox{if $b = \ell - 1$ and $\ell > 2$,}\\
\{V_{0,\ell}\},&
\mbox{if $b = 1$, $\chi_1\chi_2^{-1}=\omega$ and $\rho|_{G_{\Q_\ell}}$
 is tr\`es ramifi\'ee,}\\
\{V_{0,\ell},V_{0,1},V_{1,\ell-2}\},&
\mbox{if $b = 1$, $\ell>3$ and $\rho|_{G_{\Q_\ell}}$ is split,}\\
\{V_{0,3},V_{0,1},V_{1,3},V_{1,1}\},&
\mbox{if $b = 1$, $\ell=3$ and $\rho|_{G_{\Q_\ell}}$ is split,}\\
\{V_{0,\ell},V_{0,1}\},&
\mbox{otherwise.}\end{array}\right.$$

Propositions~\ref{prop:classical} and Corollary~\ref{gl1}(2) show that 
$\rho$ is modular of weight\footnote{Most of the literature on Serre's
Conjecture in the classical case uses arithmetic conventions, so for
the purpose of this proof, we view $\rho$ as ``modular of weight $k$''
if $\rho \sim \ol{\rho}_\pi$ for some cuspidal automorphic $\pi$ with
$\pi_\infty\cong D_{k,k}$ in the notation of \S2 and \cite{carayol:ens}.}
$b+1$ and level prime to $\ell$
if and only if $\omega^{a}\rho$ is
modular of weight $V_{a,b}$. If $\rho$ is modular
of weight $2$ and level prime to $\ell$, then $\rho|_{G_{\Q_\ell}}$ arises from a finite
flat group scheme over $\Z_\ell$, so it follows from results
of Deligne and Fontaine (\cite{edixhoven}, Theorems~2.5, 2.6) and the explicit
description of $W(\rho)$ above that
if $\rho$ is modular of weight $V$, then $V \in W(\rho)$.

To show that if $V \in W(\rho)$, then $\rho$ is
modular of weight $V$, we combine the following four
results. The first of these is a standard consequence of
multiplication by the Hasse invariant (or Eisenstein series)
of weight $\ell - 1$.
\begin{lemma} \label{lem:Eis} If $\rho$ is modular of weight $2$ and
level prime to $\ell$, then $\rho$ is modular of weight
$\ell + 1$ and level prime to $\ell$.
\end{lemma}

The theorem in the irreducible case is then a consequence of the following
result of Edixhoven; see the second paragraph of \cite[4.5]{edixhoven}.
\begin{lemma} \label{lem:theta} Suppose that $\rho$ is
modular of weight $b+1$ and level prime to $\ell$ with $2 \le b \le \ell$.
If $\rho|_{G_{\Q_\ell}}$ is irreducible, then $\omega^{1-b}\rho$
is modular of weight $\ell+2-b$ and level prime to $\ell$.
\end{lemma}

To treat the reducible case, we first apply Mazur's Principle
\cite[2.8]{edixhoven}.
\begin{theorem}\label{thm:mazur} Suppose that $\rho$ is modular
of weight $\ell + 1$ and level prime to $\ell$. If $\rho$ is not
modular of weight $2$ and level prime to $\ell$, then 
$\rho|_{G_{\Q_\ell}}$ is a tr\`es ramifi\'ee representation of the
form $\begin{pmatrix}\omega\chi_2&{*}\\0&\chi_2\end{pmatrix}$
for some unramified character $\chi_2$.
\end{theorem}

The theorem is then a consequence of the companion forms theorem
of Gross \cite{gross}\footnote{Gross's proof relies on the Hecke-equivariance
of certain isomorphisms, later checked by Cais in \cite{cais_thesis}.
In the meantime, other proofs were given by Coleman and Voloch \cite{cv} 
(but under slightly different hypotheses than we need), Faltings
and Jordan \cite{fj}, and Gee \cite{gee:duke}}: 
\begin{theorem}\label{thm:companion}
Suppose that $\rho$ is modular of weight $b+1$ and level
prime to $\ell$ with $1 \le b \le \ell - 2$. If $\rho|_{G_{\Q_\ell}}$
is reducible and split, then $\omega^{-b}\rho$ is modular
of weight $\ell - b$ and level prime to $\ell$.
\end{theorem}
Theorem~\ref{thm:Qconj} now follows.
\end{proof}
We end by remarking that Edixhoven's refinement of Serre's conjecture
includes the statement that if $\rho$ is unramified at~$\ell$ then it
should come from a mod~$\ell$ modular form of weight~1. This refinement
is not implied by our conjecture.

\section{Mod $\ell$ Langlands correspondences}
\label{sec:refined}
As Serre himself remarks in \cite{serre:duke}, his conjecture can be viewed
as part of a ``mod~$\ell$ Langlands philosophy.'' Indeed the weak conjecture
can be viewed as asserting the existence of a global mod $\ell$ Langlands
correspondence for $\GL_2/\Q$, and the refinement can be viewed as a
local-global compatibility statement. This was made precise by Emerton
in \cite{emerton:draft-a} as follows.
Consider the representation
$$H = \lim_{\stackrel{\rightarrow}{U}} H^1_\et(Y_{U,\overline{K}},\Flbar)$$
of $G_\Q\times \GL_2(\A_f)$. The weak form of Serre's conjecture
is the statement that if
 $\rho:G_\Q\to\GL_2(\Flbar)$ is continuous, odd and irreducible, then
$$\pi(\rho) := \Hom_{\Flbar[G_\Q]}(\rho,H)$$
is non-zero. Under some technical hypotheses on $\rho|_{G_{\Q_\ell}}$,
Emerton shows that $\pi(\rho)$ factors as a restricted tensor
product of representations $\pi_p$ where $\pi_p$ is a representation
of $\GL_2(\Q_p)$ determined by $\rho|_{G_{\Q_p}}$. The ``level part''
of Serre's refinement can then be recovered from the fact that
if $p\neq \ell$ and $\rho|_{G_{\Q_p}}$ has Artin conductor $p^{c_p}$, then
$\pi_p^{U_1(p^{c_p})}\neq 0$, and the ``weight part'' from the
fact that 
$$\Hom_{\GL_2(\Z_\ell)}(\det{}^{1-k}\otimes\Symm^{k-2}\Flbar^2,\pi_\ell)\neq 0$$
where $k= k(\rho)\ge2$ is the weight of $\rho|_{G_{\Q_\ell}}$ as defined
by Serre in \cite{serre:duke}.
In this section we formulate a conjectural extension of Emerton's refinement to
our setting, namely that of mod $\ell$ representations arising from a quaternion
algebra $D$ over a totally real field $K$. 

Suppose now that
$$\rho:G_K\to \GL_2(\Flbar)$$
is continuous, totally odd and irreducible.
We shall associate to $\rho$ a smooth
representation $\pi^D(\rho)$ of $(D\otimes\hat{\Z})^\times$ over $\Flbar$
and give a conjectural description for it as a product of local factors.
Thus for each prime $\gp$ of $K$ we would like to associate
to $\rho$ a smooth admissible representation of $D_\gp^\times$
defined over $\Flbar$, ideally depending only on $\rho|_{G_{K_\gp}}$.
We are able to achieve this for primes
$\gp$ not dividing $\ell$;
indeed this was already done by Emerton if $D$ is split at $\gp$
and the main new ingredient of this section is to treat the case
where $D_\gp$ is a quaternion algebra. 
For $\gp|\ell$, we are not yet able to give a
conjectural description of the local factor, but 
the weight conjecture formulated in the preceding section
can be interpreted as a description of the Jordan-H\"older
factors of its socle
under a maximal compact subgroup of $D_\gp^\times$. (We are grateful to 
Breuil and
Emerton for this observation.)

We begin by recalling Emerton's formulation
of a mod $\ell$ local Langlands correspondence for $\GL_2$ over $K_\gp$
for $\gp$ not dividing $\ell$; see 
\cite{emerton:draft-b}.
Emerton's construction is a modification
of one provided by Vign\'eras in \cite{vigneras:gln}, on whose
results it relies, a key difference being that \cite{emerton:draft-b}
involves reducible representations of $\GL_2(K_\gp)$ in order 
to prove local-global compatibility.

Fix for now a prime $\gp$ not dividing $\ell$ and let $q = \N(\gp) = \#(\CO_K/\gp)$.
For a continuous representation
$\tilde{\rho}:G_{K_\gp}\to \GL_2(\Qlbar)$, we let $\pi(\tilde{\rho})$
denote its local Langlands correspondent as modified in \cite{emerton:coates}
or \cite{emerton:draft-b}.
More precisely, $\pi(\tilde{\rho})$ is the 
usual\footnote{Recall that we follow Carayol's conventions in \cite{carayol:ens} with regard to the local Langlands correspondence; these differ from Emerton's, but we make the same modification as in \cite{emerton:coates, emerton:draft-b}.}
irreducible admissible $\Qlbar$-representation of $\GL_2(K_\gp)$ provided by
the local Langlands correspondence, unless $\tilde{\rho}$ is the
sum of two characters whose ratio is cyclotomic in which case
$\pi(\tilde{\rho})$ is a generic representation of length two.
\begin{theorem} \label{thm:emerton} (Emerton) There is a map $\rho \mapsto \pi(\rho)$ from the
set of isomorphism classes of continuous representations $G_{K_\gp}\to\GL_2(\Flbar)$
to the set of isomorphism classes
of finite-length smooth admissible $\Flbar$-representations of $G=GL_2(K_\gp)$,
uniquely determined by the following properties:
\begin{enumerate}
\item The representation $\pi(\rho)$ has a unique irreducible subrepresentation $\pi'$, and
$\pi(\rho)/\pi'$ is finite-dimensional.
\item If $\tilde{\rho}:G_{K_\gp} \to \GL_2(\Qlbar)$ is a continuous lift of $\rho$,
then there is a $G$-equivariant $\ol{\Z}_\ell$-lattice in $\pi(\tilde{\rho})$
whose reduction admits a $G$-equivariant embedding into $\pi(\rho)$. 
\item If $\pi$ is a finite-length smooth admissible $\Flbar$-representation of $G$
satisfying conditions (1) and (2), then there exists a $G$-equivariant
embedding $\pi \hookrightarrow \pi(\rho)$.
\end{enumerate}
\end{theorem}
If $D$ is split at $\gp$, we let $\pi^{D_\gp}(\rho)$ denote the
representation of $D_\gp^\times\cong\GL_2(K_\gp)$ given by the
theorem. 

We would like an analogue of the theorem which associates to $\rho$
a representation of $D_\gp^\times$ when $D_\gp$ is a non-split
quaternion algebra
over $K_\gp$. Again our construction is a modification of one
already provided by Vign\'eras (this time in \cite{vigneras:quat})
on whose results we rely.
Since the irreducible smooth admissible $\Flbar$-representations
of $D_\gp^\times$ are finite-dimensional,
one might expect a result like Theorem~\ref{thm:emerton} but without condition (1).
However results of Ribet \cite{ribet:psvolume} show this naive generalisation 
would not satisfy local-global compatibility,
and there are further complications when $\ell=2$. 
On the other hand we can give a more explicit description of the desired representation
$\pi^{D_\gp}(\rho)$.

Recall from Theorem~15.1 of~\cite{JL}
that the (local) Jacquet-Langlands correspondence establishes a bijection between
isomorphism classes of irreducible admissible square-integrable $\Qlbar$-representations of
$\GL_2(K_\gp)$ and irreducible admissible $\Qlbar$-representations of $D_\gp^\times$.
If $\tilde{\rho}:G_{K_\gp} \to \GL_2(\Qlbar)$ is a continuous representation, then
$\pi(\tilde{\rho})$ is square-integrable if and only if $\tilde{\rho}$ is either
irreducible or isomorphic to a twist of the non-split representation of the form
$\smallmat{1}{*}{0}{\omega^{-1}}$ where $\omega$ is the $\ell$-adic cyclotomic character,
and then we let $\pi^{D_\gp}(\tilde{\rho})$ the representation of $D_\gp^\times$
corresponding to $\pi(\tilde{\rho})$ via Jacquet-Langlands.

Now consider again a continuous representation $\rho:G_{K_\gp}\to\GL_2(\Flbar)$.
It is straightforward to check that there exist continuous representations 
$\tilde{\rho}:G_{K_\gp}\to\GL_2(\Qlbar)$ containing $G_{K_\gp}$-stable lattices
whose reduction is $\rho$. Moreover, if $\tilde{\rho}$ is irreducible, then so is $\rho$
except in the following case: if $q\equiv-1\bmod\ell$, $L$ is the quadratic unramified
extension of $K_\gp$ and $\tilde{\rho}\sim \chi\otimes\Ind_{G_L}^{G_{K_\gp}}\xi$ for
some characters $\chi$ of $G_{K_\gp}$ and $\xi$ of $G_L$ such that $\xi$ is different
from its $\Gal(L/K_\gp)$-conjugate and has $\ell$-power order, then
$\rho$ has semi-simplification isomorphic to $\ol{\chi}\oplus \ol{\chi}\ol{\omega}^{-1}$ where
$\ol{\omega}$ is the mod $\ell$ cyclotomic character. (Note that if
$q\equiv-1\bmod\ell$, then $\ol{\omega}$ is the quadratic unramified character
of $G_{K_\gp}$, unless $\ell=2$ in which case $\ol{\omega}=1$.)

For representations of $D_\gp^\times$ we have the following result of Vign\'eras
(Propositions~9 and~11 and Corollary~12 of \cite{vigneras:quat}):
\begin{proposition}(Vign\'eras) \label{prop:vigneras} Suppose that $\rho:G_{K_\gp}\to\GL_2(\Flbar)$
is continuous and irreducible. If $\tilde{\rho}:G_{K_\gp}\to \GL_2(\Qlbar)$
is a lift of $\rho$, then the reduction of $\pi^{D_\gp}(\tilde{\rho})$ is
irreducible and depends only on $\rho$.
\end{proposition}
If $\rho:G_{K_\gp}\to\GL_2(\Flbar)$ is irreducible, we define $\pi^{D_\gp}(\rho)$
to be the reduction of $\pi^{D_\gp}(\tilde{\rho})$ for any lift $\tilde{\rho}$
of $\rho$; this is well-defined by the proposition.

Suppose now that $\rho$ is reducible. If $\rho$ is not a twist of a representation of
the form $\smallmat{1}{*}{0}{\ol{\omega}^{-1}}$ (where $\ol{\omega}$ is the mod $\ell$ cyclotomic character),
then $\rho$ has no lifts $\tilde{\rho}$ such that $\pi(\tilde{\rho})$ is square-integrable,
and we define $\pi^{D_\gp}(\rho)=0$.

Suppose now that $\rho\sim\smallmat{\chi}{*}{0}{\chi\ol{\omega}^{-1}}$ for some character
$\chi:G_{K_\gp}\to\Flbar^\times$ (i.e., $\rho$ is any extension of $\chi\ol{\omega}^{-1}$ by $\chi$).
If $q\not\equiv-1\bmod\ell$, then the only
lifts $\tilde{\rho}$ of $\rho$ for which $\pi(\tilde{\rho})$ is square-integrable
are non-split representations of the form $\smallmat{\tilde{\chi}}{*}{0}{\tilde{\chi}\omega^{-1}}$
where $\tilde{\chi}:G_{K_\gp}\to\CO_{K_\gp}^\times$ lifts $\chi$.
In this case $\pi^{D_\gp}(\tilde{\rho}) = \tilde{\chi}^{-1}\circ\det$ where $\det:D_\gp^\times
\to K_\gp^\times$ is the reduced norm (using $\tilde{\chi}$ and $\chi$ also to denote the
characters of $K_\gp^\times$ to which they correspond via class field theory). 
We then define $\pi^{D_\gp}(\rho)$ to be $\chi^{-1}\circ\det$, {\em unless}
$q\equiv 1 \bmod\ell$ and $\rho$ is the split representation $\chi\oplus\chi$
in which case we define $\pi^{D_\gp}(\rho)$ to be $(\chi^{-1}\circ\det)\oplus (\chi^{-1}\circ\det)$
(note that $\ol{\omega}$ is trivial).
\begin{remark}
Note that the reduction of $\pi^{D_\gp}(\tilde{\rho})$ is $\chi^{-1}\circ\det$, which then
coincides with $\pi^{D_\gp}(\rho)$ unless we are in the exceptional case where $q\equiv1\bmod\ell$
and $\rho$ is scalar.  In this case our definition of $\pi^{D_\gp}(\rho)$ is motivated
by results of Ribet~\cite{ribet:psvolume} and Yang~\cite{yang} on multiplicities of
Galois representations in the cohomology of Shimura curves.
\end{remark}

Suppose now that $\rho\sim\smallmat{\chi}{*}{0}{\chi\ol{\omega}^{-1}}$ for some character
$\chi:G_{K_\gp}\to\Flbar^\times$ and that $q\equiv-1\bmod\ell$. Let $c_\rho$ denote
the extension class associated to $\rho$ in 
$$\Ext^1_{\Flbar[G_{K_\gp}]}(\chi\ol{\omega}^{-1},\chi)\cong H^1(G_{K_\gp},\Flbar(\ol{\omega})).$$
Note that this space is $1$-dimensional unless $\ell=2$ in which case it is $2$-dimensional.
Recall that such $\rho$ have irreducible lifts $\tilde{\rho}$ which are twists of tamely ramified
representations induced from $G_L$ where $L$ is the unramified quadratic extension of $K_\gp$. For such
$\tilde{\rho}$, $\pi^{D_\gp}(\tilde{\rho})$ is a two-dimensional representation of $D_{\gp}^\times$
whose reduction (which depends on a choice of lattice) has semi-simplification 
$\chi^{-1}\circ\det\oplus (\chi^{-1}\ol{\omega}^{-1})\circ\det$ (see \cite{vigneras:quat}). 
We will define $\pi^{D_\gp}(\rho)$ as a certain extension of
$(\chi^{-1}\ol{\omega}^{-1})\circ\det$ by $\chi^{-1}\circ\det$ depending on $c_\rho$.
To this end we will first compute 
$$\Ext^1_{\Flbar[D_\gp^\times]}((\chi^{-1}\ol{\omega}^{-1})\circ\det,\chi^{-1}\circ\det)
 \cong H^1(D_\gp^\times,\Flbar(\ol{\omega}\circ\det)).$$
Although a unified treatment is possible (see Remark~\ref{rmk:unified}),
it is simpler to consider separately the cases $\ell > 2$ and $\ell = 2$.

Suppose first that $\ell > 2$. Let $\CO_{D_\gp}$ denote the maximal
order in $D_\gp$ and $\Pi$ a uniformizer, so $\val_{K_\gp}(\det\Pi) = 1$
and $\CO_{D_\gp}/\Pi\CO_{D_\gp} \cong \F_{q^2}$. Letting
$\Gamma = D_\gp^\times/(1+\Pi\CO_{D_\gp})$, we have an exact sequence
$$1 \to \F_{q^2}^\times \to \Gamma \to \Z \to 0$$
where the map $\Gamma\to\Z$ is $\val\circ\det$ and $n\in \Z$
acts on $\F_{q^2}^\times$ by $x\mapsto x^{q^n}$.
Note that $\F_{q^2}^\times$ acts trivially on $\Flbar(\ol{\omega}\circ\det)$
and the induced action of $n\in\Z$ is via the character $\mu(n) = (-1)^n=q^n$.
Since $1+\Pi\CO_{D_\gp}$ is pro-$p$, we have that
$$H^1(D_\gp^\times,\Flbar(\ol{\omega}\circ\det))
 \cong H^1(\Gamma,\Flbar(\ol{\omega}\circ\det)).$$
Since $H^1(\Z,\Flbar(\mu))= H^2(\Z,\Flbar(\mu))=0$, we have that
$$ H^1(\Gamma,\Flbar(\ol{\omega}\circ\det))=\Hom_\Z(\F_{q^2}^\times,\Flbar(\mu))$$
is one-dimensional. Hence there is a unique isomorphism
class of $\Flbar$-representations of $D_\gp^\times$ which
are non-trivial extensions of 
$(\chi^{-1}\ol{\omega}^{-1})\circ\det$ by $\chi^{-1}\circ\det$.
We define $\pi^{D_\gp}(\rho)$ to be the extension which is
split if and only if $c_\rho$ is trivial.

Finally consider the case $\ell=2$. Then $\ol{\omega}$ is trivial and
$$H^1(D_\gp^\times,\ol{\F}_2) = \Hom(D_\gp^\times/(D_\gp^\times)^2,\ol{\F}_2),$$
and one checks easily that $\det$ induces an isomorphism 
$$D_\gp^\times/(D_\gp^\times)^2 \iso K_\gp^\times/(K_\gp^\times)^2.$$
On the other hand local class field theory yields an isomorphism
$$H^1(G_{K_\gp},\ol{\F}_2) \cong \Hom(K_\gp^\times/(K_\gp^\times)^2,\ol{\F}_2).$$
Putting these isomorphisms together yields 
$$H^1(G_{K_\gp},\ol{\F}_2) \iso H^1(D_\gp^\times,\ol{\F}_2),$$
and we define $\pi^{D_\gp}(\rho)$ to be the extension obtained from
the image of $c_\rho$.

\begin{remark} \label{rmk:unified} To give a unified treatment for the cases
$\ell = 2$ and $\ell > 2$ when $q\equiv -1\bmod\ell$ and $\rho$ as above, embed
the unramified quadratic extension $L$ of $K_\gp$ in $D_\gp$ and
let $N$ denote the normaliser of the image of $L^\times$ in $D_\gp^\times$.
One can then check that restriction induces an isomorphism
$$H^1(D_\gp^\times,\Flbar(\ol{\omega}\circ\det)) \cong H^1(N,\Flbar(\mu))$$
where $\mu(x)=1$ if $x\in L^\times$ and $\mu(x)=-1$ otherwise.
On the other hand one finds that $N$ is isomorphic by local class
field theory to the image of the Weil 
group of $K_\gp$ in $\Gal(L^\ab/K_\gp)$ and that the inflation and
restriction maps induce isomorphisms
$$H^1(G_{K_\gp},\Flbar(\ol{\omega}))\cong
 H^1(\Gal(L^\ab/K_\gp),\Flbar(\ol{\omega}))
 \cong H^1(N,\Flbar(\mu)).$$
\end{remark}

\begin{remark} It is straightforward to check 
that if $q\equiv -1 \bmod\ell$ and $\xi$ is a character of $G_L$ of
$\ell$-power order and not equal to its Galois conjugate, then every representation $\rho$ with semisimplification
$1\oplus \ol{\omega}$ is isomorphic to the reduction of a $G_{K_\gp}$-stable
lattice in $\Ind_{G_L}^{G_{K_\gp}}\xi$. It follows that if $\ell=2$, then
every representation with scalar semisimplification has the same set
of lifts $\tilde{\rho}$ such that $\pi(\tilde{\rho})$ is square-integrable,
so $\pi^{D_\gp}(\rho)$ is not characterised by the set of
$\pi^{D_\gp}(\tilde{\rho})$.
(Note that this is also the case for trivial reasons if 
$q\equiv 1\bmod\ell$ and $\ell>2$.)
\end{remark}

We now return to the global setting of a totally odd, continuous, irreducible
$$\rho:G_K\to\GL_2(\Flbar)$$
and construct the $\Flbar$-representation of $D_f^\times = (D\otimes\hat{\Z})^\times$ whose
local factors should be the $\pi^{D_\gp}(\rho)$. We first consider the
case of a totally definite quaternion algebra $D$ over $K$. 
Fix a maximal order $\CO_D$ in $D$ and isomorphisms
$\CO_{D,\gp} \cong \M_2(\CO_{K,\gp})$ for each prime $\gp$ of $K$ at
which $D$ is split.

For each open compact subgroup $U$ of $D_f^\times$ we define
$$S^D(U) = \{\,f:D^\times\backslash D_f^\times/U \to \Flbar \,\}.$$
The obvious projection maps for $V \subset U$ induce inclusions
$S^D(U) \to S^D(V)$ and we define $S^D$ as the direct limit of
the $S^D(U)$. Thus $S^D$ can equivalently be defined as the
set of smooth functions $f:D^\times\backslash D_f^\times \to \Flbar$.
Note that this $\Flbar$-vector space admits a natural left action of
$D_f^\times$ by right translation, and $S^D(U) = (S^D)^U$, the $U$-invariant
functions in $S^D$.
Moreover for any $g\in D_f^\times$ and open compact $U,V\subset D_f^\times$
we have the double coset operator $[UgV]: S^D(V) \to S^D(U)$ defined
in the usual way. In particular, for each prime $\gp$ at which $D$
is split and $\GL_2(\CO_{K,\gp})\subset U$, we have the endomorphisms
$$T_\gp = \left[ U \smallmat{\varpi_\gp}{0}{0}{1} U \right]\quad \mbox{and}\quad
 S_\gp = \left[ U \smallmat{\varpi_\gp}{0}{0}{\varpi_\gp} U \right]$$
of $S^D(U)$, where $\varpi_\gp$ is any uniformizer of $\CO_{K,\gp}$.
If $\Sigma$ is a finite set of primes of $K$ containing all those
such that:
\begin{itemize}
\item $D$ is ramified at $\gp$,
\item $\GL_2(\CO_{K,\gp})\not\subset U$,
\item $\rho$ is ramified at $\gp$, or
\item $\gp$ divides $\ell$,
\end{itemize}
then we let $\T^\Sigma(U)$ denote the commutative $\Flbar$-subalgebra of
$\End_{\Flbar}(S^D(U))$ generated by the $S_\gp$ and $T_\gp$ for $\gp\not\in\Sigma$.
We let $\gm_\rho^\Sigma = \gm_\rho^\Sigma(U)$ denote the ideal of
$\T^\Sigma(U)$ generated by the operators
$$T_\gp - S_\gp\tr(\rho(\frob_\gp))\quad\mbox{and}\quad \N(\gp) - S_\gp\det(\rho(\frob_\gp))$$
for all $\gp\not\in\Sigma$. We let
$$S^D(U)[\gm_\rho^\Sigma] = \{\,f\in S^D(U)\,|\, \mbox{ $Tf= 0$ for all $T\in\gm_\rho^\Sigma$}\,\}.$$
If $\rho = \ol{\rho}_\pi$ for some (necessarily cuspidal) automorphic representation $\pi$
of $D^\times$ with weight $(\vec{2},0)$
and $\pi^U\neq 0$, then $\gm_\rho^\Sigma$ is a maximal ideal of $\T^\Sigma(U)$
and $S^D(U)[\gm_\rho^\Sigma]\neq 0$; otherwise $\gm_\rho^\Sigma = \T^\Sigma(U)$
and $S^D(U)[\gm_\rho^\Sigma] =0$.

\begin{lemma} \label{lemma:Hecke}
Suppose that $D$, $U$, $\rho$ and $\Sigma$ are as above. Then\\
a) $S^D(U)[\gm_\rho^\Sigma]$ is independent of $\Sigma$ (so we will denote it
 $S^D(U)[\gm_\rho]$);\\
b) if $g\in D_f^\times$ and $V$ is an open compact subgroup of $D_f^\times$
 such that $V \subset gUg^{-1}$, then $g$ sends $S^D(U)[\gm_\rho]$ to $S^D(V)[\gm_\rho]$.
\end{lemma}

\begin{proof} a) We may assume $\Sigma' = \Sigma\cup\{\gp\}$ for some $\gp\not\in\Sigma$ and 
that $\gm_\rho^{\Sigma'}\neq \T^{\Sigma'}(U)$. Since $\rho$ is irreducible, a standard argument
using the representations $\rho_\pi$ lifting $\rho$ gives a representation
$$\rho':G_K \to \GL_2(\T^{\Sigma'}(U)_{\gm_\rho^{\Sigma'}})$$
lifting $\rho$ such that 
$$T_\gp = S_\gp\tr(\rho'(\Frob_\gp))\quad\mbox{and}\quad
 \N(\gp) = S_\gp\det(\rho'(\Frob_\gp))$$
as endomorphisms of $S^D(U)_{\gm_\rho^{\Sigma'}}$. It follows
that $S^D(U)[\gm_\rho^{\Sigma}] = S^D(U)[\gm_\rho^{\Sigma'}]$.

b) Choosing $\Sigma$ sufficiently large that $\GL_2(\CO_{K,\gp})\subset V$
and $g_\gp \in \GL_2(\CO_{K,\gp})$ for all $\gp\not\in\Sigma$, we see that
$g$ commutes with $T_\gp$ for all $\gp\not\in\Sigma$.
\end{proof}

We can now consider the direct limit over $U$ of the spaces
$S^D(U)[\gm_\rho]$; by the lemma, this makes sense and yields a
representation $S^D[\gm_\rho]$ of $D_f^\times$.  We can now state
an analogue of Emerton's local-global compatibility 
conjecture~\cite{emerton:draft-a}.
\begin{conjecture} \label{conj:definite} Suppose that $K$ is a totally real field,
$$\rho:G_K\to\GL_2(\Flbar)$$
is a continuous, irreducible and totally odd representation,
and $D$ is a totally definite quaternion algebra over $K$.
Then the $\Flbar$-representation $S^D[\gm_\rho]$ of $D_f^\times$ is isomorphic to
a restricted tensor product
$\otimes'_{\gp} \pi_\gp$ where $\pi_\gp$ is a smooth admissible representation
of $D_\gp^\times$ such that
\begin{itemize}
\item if $\gp$ does not divide $\ell$, then 
$\pi_\gp \cong \pi^{D_\gp}(\rho|_{G_{K_\gp}})$;
\item if $\gp|\ell$ then $\pi_\gp\neq 0$; moreover if $K$ and $D$ are unramified at $\gp$, and
$\sigma$ is an irreducible $\Flbar$-representation of
 $\GL_2(\CO_{K,\gp})$, then $\Hom_{\GL_2(\CO_{K,\gp})}(\sigma, \pi_\gp) \neq 0$ if and only
 if $\sigma \in W_\gp(\rho^\vee|_{G_{K_\gp}})$.
\end{itemize}
\end{conjecture}

\begin{remark} In the case $K=\Q$ and $D=M_2(\Q)$ and $\rho|_{G_{\Q_\ell}}$ is not a twist
of a representation of the form $\smallmat{1}{*}{0}{\ol{\omega}}$, then
Emerton predicts the precise form for $\pi_\ell$ in \cite{emerton:draft-a} as well
and goes on to prove his conjecture under technical hypotheses (including the assumption
that $\ell > 2$).  It is reasonable
to expect that $\pi_\gp$ is of the form predicted there whenever 
$D_\gp \cong M_2(\Q_\ell)$ and $\rho|_{G_{K_\gp}}$ is as above.

Under the hypotheses that $K_\gp$ is an unramified extension of $\Q_\ell$ and $D$ is split
at $\gp$, the relation with $W_\gp$ can be viewed as a description of the $\GL_2(\CO_{K,\gp})$-socle
of $\pi_\gp$ (which in most cases is expected to be multiplicity-free). 
In many cases, Breuil and Paskunas~\cite{breuil_paskunas} construct infinitely many 
(isomorphism classes of) representations with the desired socle.  These representations
are irreducible whenever $\rho|_{G_{K_\gp}}$ is, and are expected to have finite length
in general.  However, the multitude of possibilities raises the question
of whether one should still expect $\pi_\gp$ to be completely determined by
$\rho|_{G_{K_\gp}}$.

Some work has also been done towards defining a conjectural set of weights
$W_\gp$ when $F_\gp$ is a ramified extension of $\Q_\ell$ and $D_\gp$ is split.
In particular, Schein~\cite{schein:ramified} gives a definition of $W_\gp$
when $\rho|_{G_{K_\gp}}$ is tamely ramified, and Gee~\cite{gee:automorphic}
gives a more general but less explicit definition than ours.

In the case where $D_\gp$ is not split, Gee and Savitt \cite{gs_new} have described
the set of weights, again in a less explicit form in general.  We remark however
that one can show in this case that $\pi_\gp$, if it exists, has infinite length.
\end{remark}

Suppose now that $D$ is split at exactly one infinite place.
We exclude the case $D=\M_2(\Q)$ already considered by Emerton.
We now define $S^D(U) = H^1(Y_{U,\bar{K}},\Flbar)$ and $S^D = \lim S^D(U)$,
the limit taken over all open compact $U\subset D_f^\times$
with respect to the maps on cohomology induced by the natural
projections $Y_V \to Y_U$ for $U\subset V$. If $V \subset gUg^{-1}$,
then we have a $K$-morphism $Y_V \to Y_U$ corresponding to
right multiplication by $g$ on complex points, inducing a homomorphism
$S^D(U)\to S^D(V)$ which we also denote by $g$, making $S^D$
a $G_K\times D_f^\times$-module. However the natural map
$S^D(U) \to (S^D)^U$ is not necessarily an isomorphism.

For $g\in D_f^\times$, and $U$,$V$ open compact subgroups of $D_f^\times$,
we have the double coset operator $[VgU]:S^D(U) \to S^D(V)$ defined as the
composite $S^D(U) \to S^D(V') \to S^D(V)$ where $V' = V \cap gUg^{-1}$,
the first map is defined by $g$, and the second is the trace morphism
times the integer $[V:V']/\deg(Y_{V'}/Y_V)$. We can thus define Hecke operators
$T_\gp$ for $\gp\not\in\Sigma$, algebras $\T^\Sigma(U)\subset \End_{\Flbar}(S^D(U))$, ideals
$\gm_\rho^\Sigma$ and subspaces $S^D(U)[\gm_\rho^\Sigma]\subset S^D(U)$,
just as in the case of totally definite $D$. 
The analogue of Lemma~\ref{lemma:Hecke} is proved in exactly the same way,
now yielding a representation $S^D[\gm_\rho]$ of $G_K \times D_f^\times$.

\begin{conjecture} \label{conj:indefinite} Suppose that $K$ is a totally real field,
$$\rho:G_K\to\GL_2(\Flbar)$$
is a continuous, irreducible and totally odd representation,
and $D$ is a quaternion algebra over $K$ split at exactly one real place.
Then the $\Flbar$-representation $S^D[\gm_\rho]$ of $G_K\times D_f^\times$ is isomorphic to
$\rho\otimes \left(\otimes'_{\gp} \pi_\gp\right)$ where $\pi_\gp$ is a smooth admissible representation
of $D_\gp^\times$ such that
\begin{itemize}
\item if $\gp$ does not divide $\ell$,
then $\pi_\gp \cong \pi^{D_\gp}(\rho|_{G_{K_\gp}})$;
\item if $\gp|\ell$ then $\pi_\gp\neq 0$; moreover if $K$ and $D$ are unramified at $\gp$, and
$\sigma$ is an irreducible $\Flbar$-representation of
 $\GL_2(\CO_{K,\gp})$, then $\Hom_{\GL_2(\CO_{K,\gp})}(\sigma, \pi_\gp) \neq 0$ if and only
 if $\sigma \in W_\gp(\rho^\vee|_{G_{K_\gp}})$.
\end{itemize}
\end{conjecture}

By the following lemma, Conjecture~\ref{conj:indefinite} could be reformulated
as saying that
the representation $\Hom_{\Flbar[G_K]}(\rho, S^D)$ of $D_f^\times$ has the
prescribed form, as in \cite{emerton:draft-a}.
\begin{lemma} \label{lemma:esblr}
The evaluation map $\rho\otimes_{\Flbar}\Hom_{\Flbar[G_K]}(\rho,S^D)\to S^D$
induces a $G_K\times D_f^\times$-linear isomorphism:
$$\rho \otimes_{\Flbar}\Hom_{\Flbar[G_K]}(\rho,S^D) \iso S^D[\gm_\rho].$$
\end{lemma}
\begin{proof} 
It suffices to prove the lemma with $S^D$ replaced by $S^D(U)$ and take direct
limits. Since $\rho$ is irreducible, the evaluation map
$$\rho\otimes_{\Flbar}\Hom_{\Flbar[G_K]}(\rho,S^D(U))\to S^D(U)$$
is injective by Schur's Lemma.
Using the Eichler-Shimura relations on $Y_U$ (in particular, that $\Frob_\gp^2 - S_\gp^{-1}T_\gp\Frob_\gp
+ \N(\gp)S_\gp^{-1} = 0$ on $H^1(Y_{U,\ol{K}},\Flbar)$ for all $\gp\not\in\Sigma$), one shows
as in the proof of Prop.~3.2.3 of \cite{breuil_emerton} that the image lies
in $S^D(U)[\gm_\rho]$. Finally, the main result of \cite{blr} shows that
$S^D(U)[\gm^\rho]$ is isomorphic to a direct sum of copies of $\rho$, hence
the map is surjective.
\end{proof}

Next we show how one can usually recover $S^D(U)[\gm_\rho]$ from
$S^D[\gm_\rho]$. The
caveat (an observation going back to Ribet) is
that it is not quite true that ``all errors are Eisenstein''. Let
us say that a representation $\rho$ is {\em badly dihedral} if $\rho$
is induced from a character of $G_{K'}$ where $K'$
is a totally imaginary
quadratic extension of $K$ of the form $K(\delta)$ for
some $\delta$ such that $\delta^\ell\in K$.
For $\ell>2$ it is not difficult to check that if $\rho$
is badly dihedral then $K$ must contain the maximal real
subfield $\Q(\mu_{\ell})^+$ of $\Q(\mu_{\ell})$ and that $K'=K(\mu_\ell)$
(because the field $K'$ contains $\zeta:=\overline{\delta}/\delta$ with
$\overline{\delta}$ the Galois conjugate of $\delta$, and $\zeta\not=1$
is an $\ell$th root of unity). In particular if $\ell>3$ and $\ell$ is
unramified in $K$, then
there will be no badly dihedral representations at all.
However for $\ell=2$ there may be more than one possibility for $K'$
(but only finitely many).

\begin{lemma} \label{lemma:elliptic}
The natural map $S^D(U)[\gm_\rho]\to (S^D[\gm_\rho])^U$ is injective;
moreover it is an isomorphism if either
\begin{itemize}
\item $Y_U$ has no elliptic points of order a multiple of $\ell$ or
\item $\rho$ is not badly dihedral.
\end{itemize}
\end{lemma}
\begin{proof} It suffices to show that
if $V$ is any normal open compact subgroup of $U$ then the natural
map
$$H^1(Y_{U,\ol{K}},\Flbar) \to H^1(Y_{V,\ol{K}},\Flbar)^{U/V}$$
is injective after localising at $\gm_\rho$, and is an isomorphism
under the additional hypotheses. Equivalently we must show
that $\rho$ does not appear in the $\Flbar[G_K]$-semisimplification
of the kernel of the above map, and under the additional hypotheses
it does not appear in the cokernel either.

Let $Z_U$ denote the reduced closed subscheme of $Y_U$ defined by
its elliptic points and let $Z_V = (Z_U\times_{Y_U}Y_V)^\red$,
$W_U = Y_U-Z_U$ and $W_V = Y_V-Z_V$. Then the morphism $W_V \to W_U$
is \'etale with Galois group $\Gamma$ (a quotient of $U/V$), and the
Hochschild-Serre spectral sequence yields an exact
sequence:
$$\begin{array}{rl}
 0\to H^1(\Gamma,H^0(W_{V,\ol{K}},\Flbar)) &\to
 H^1(W_{U,\ol{K}},\Flbar)\\ & \to H^1(W_{V,\ol{K}},\Flbar)^U
 \to H^2(\Gamma,H^0(W_{V,\ol{K}},\Flbar)).\end{array}$$
The inclusion $W_U \to Y_U$ yields an exact sequence:
$$\begin{array}{rl}
H^1_{Z_{U,\ol{K}}}(Y_{U,\ol{K}},\Flbar) \to & H^1(Y_{U,\ol{K}},\Flbar) \\ 
\to & H^1(W_{U,\ol{K}},\Flbar) \to H^2_{Z_{U,\ol{K}}}(Y_{U,\ol{K}},\Flbar) 
\end{array}$$
By the excision theorem, $H^i_{Z_{U,\ol{K}}}(Y_{U,\ol{K}},\Flbar) = \bigoplus_{z\in Z_U(\bar{K})}
 H^i_{\{z\}}(Y_{U,\ol{K}},\Flbar)$. By the Betti-\'etale comparison theorem for example, we
 see that each $H^1_{\{z\}}(Y_{U,\ol{K}},\Flbar)=0$ and that each $H^2_{\{z\}}(Y_{U,\ol{K}},\Flbar)$
 is one-dimensional; moreover if $\psi:X'\to X$ is a non-constant morphism of smooth proper
 curves over $\ol{K}$ with $\psi(x') = x$, then the induced map 
 $H^2_{\{x\}}(X,\Flbar)\to H^2_{\{x'\}}(X',\Flbar)$ is trivial if the ramification degree $e(x'/x)$ is divisible by $\ell$ and it is an isomorphism otherwise.
 In particular, if $z\in Z_U(\ol{K})$ is defined over $L$, then the morphism
 $Y_{U,L} \to \P^1_{L}$ gotten from a uniformizer at $z$ induces an isomorphism
 $$\Flbar(1) = H^2_{\{0\}}(\P^1_{\ol{L}},\Flbar) \iso H^2_{\{z\}}(Y_{U,\ol{L}},\Flbar)$$
 of $G_{L}$-modules. It follows that $H^2_{Z_{U,\ol{K}}}(Y_{U,\ol{K}},\Flbar)
 \cong \bigoplus_{P\in Z_U} \Ind_{G_{K(P)}}^{G_K}\Flbar(1)$ as $G_K$-modules.
Combining this with the corresponding exact sequence arising from $W_V\to Y_V$ yields
a commutative diagram
$$\begin{array}{cccccc}
0 \to& H^1(Y_{U,\ol{K}},\Flbar) & \to &H^1(W_{U,\ol{K}},\Flbar)& \to &H^2_{Z_{U,\ol{K}}}(Y_{U,\ol{K}},\Flbar) \\
&\downarrow&&\downarrow&&\downarrow\\
0 \to& H^1(Y_{V,\ol{K}},\Flbar)^\Gamma& \to &H^1(W_{V,\ol{K}},\Flbar)^\Gamma& \to &H^2_{Z_{V,\ol{K}}}(Y_{V,\ol{K}},\Flbar)^\Gamma
\end{array}$$
such that the kernel of the rightmost vertical map is isomorphic to the direct sum 
of the $\Ind_{G_{K(P)}}^{G_K}\Flbar(1)$ over the $P\in Y_U$ whose ramification degree
in $Y_V$ is divisible by $\ell$. 
If $P$ is such an elliptic point, then it is fixed by
some $\delta\in D$ such that $\delta$ has order $\ell$
in $D^\times/K^\times$.
The extension
$K':=K[\delta]$ is a commutative integral domain within $D$, and it is
finite over $K$ and hence a field; moreover it must be a
quadratic extension of $K$, imaginary
at our preferred infinite place $\tau_0$ since $\delta$ has isolated fixed points in $\uhp^\pm$,
 and at the other infinite
places since $K'\subset D$ and hence $K'$ splits~$D$.
The elliptic point will then be a special point for the
Shimura curve $Y_U$ with respect to the torus $\Res_{K'/\Q}(\Gm)$
and by Lemma~3.11 of~\cite{cornut-vatsal} the elliptic point
will be defined over an abelian extension of $K'$.
Now under the additional hypotheses of the lemma
it follows that $\rho$
does not appear in the $\Flbar[G_K]$-semisimplification of 
the direct sum of the $\Ind_{G_{K(P)}}^{G_K}\Flbar(1)$ as above.

Recall from Lemma~\ref{eisenstein}
that the action of $G_K$ on $H^0(W_{V,\ol{K}},\Flbar) = H^0(Y_{V,\ol{K}},\Flbar)$, hence on the
kernel and cokernel of the middle vertical map, factors through an abelian quotient. 
Finally we deduce from the snake lemma that $\rho$ does not appear in the semisimplification
of the kernel of the leftmost vertical map, nor that of the cokernel under the additional
hypotheses.
\end{proof}

Finally we record some consequences of Conjectures~\ref{conj:definite} and~\ref{conj:indefinite}.
(Recall that Conjecture~\ref{conj:indefinite} is formulated for an arbitrary totally real $K$,
but that Conjecture~\ref{conj:weight} assumes $K$ is unramified at $\ell$.)
\begin{proposition}\label{indefinite:weight}
Conjecture~\ref{conj:indefinite} implies Conjecture~\ref{conj:weight}.
\end{proposition}

\begin{proof} Since the conjecture is known for $\Q$, we can assume $K \neq \Q$. 
Now $\rho$ is modular of weight $\sigma$ if and only if $\rho(-1)$ is isomorphic to
an $\Flbar[G_K]$-subquotient of $\Hom_{\Flbar[U]}(\sigma^\vee,H^1(Y_{U',\overline{K}},\Flbar))$ for
some $D$ and $U$ as in Definition~\ref{def:modular}. This is equivalent to saying that
$\gm_{\rho(-1)}^\Sigma$ (for any $\Sigma$ at level $U'$) is (maximal and) in the support of
$\Hom_{\Flbar[U]}(\sigma^\vee,S^D(U'))$, or equivalently that
$$\Hom_{\Flbar[U]}(\sigma^\vee,S^D(U')[\gm_{\rho(-1)}])\neq 0.$$
(Note in particular that by the proof of Lemma~\ref{lemma:esblr}, ``$\Flbar G_K$-subquotient'' can be replaced with
``$\Flbar G_K$-submodule'' as claimed after Definition~\ref{def:modular}.) 

So if $\rho$ is modular of an irreducible weight $\sigma$, then 
$\Hom_{\Flbar[U]}(\sigma^\vee,S^D[\gm_{\rho(-1)}])\neq 0$
by Lemma~\ref{lemma:elliptic}. If Conjecture~\ref{conj:indefinite} holds, then
we may write $S^D[\gm_{\rho(-1)}] = \rho \otimes (\otimes' \pi_\gp)$; moreover
for each $\gp|\ell$, we have $\Hom_{\Flbar[U_\gp]}(\sigma_\gp^\vee,\pi_\gp)\neq 0$
so that $\sigma_\gp^\vee \in W_\gp(\rho(-1)^\vee)$, or equivalently that
$\sigma^\vee \in W(\rho(-1)^\vee)$.
Since $\rho \cong \det(\rho)\otimes \rho^\vee$ and $\sigma \cong \psi\sigma^\vee$
where $\psi$ is the central character of $\sigma$, it follows easily from
Proposition~\ref{prop:chars} that this is equivalent to 
$\sigma\in W(\rho)$.

Conversely suppose that $\sigma\in W(\rho)$. If $[K:\Q]$ is odd, then let $D$
be a quaternion algebra over $K$ ramified at all but one infinite places
and split at all finite places. If $[K:\Q]$ is even, then let $L$ denote the
splitting field of $\rho$ and choose a prime $\gq$ unramified in $L(\mu_\ell)$
so that the conjugacy class of $\frob_\gq$ in $\Gal(L(\mu_\ell)/K)$ is
that of a complex conjugation. Let $D$ be a quaternion algebra over $K$
ramified at exactly $\gq$ and all but one infinite place.
Then $\pi^{D_\gp}(\rho)\neq 0$ for all primes $\gp$, so we can choose $U$ sufficiently
small (of level prime to $\ell$) so that $Y_U$ has no elliptic points and
$\pi_\gp^{U_\gp}\neq 0$ for all $\gp$ not dividing $\ell$.  We can then
reverse the above argument to conclude that $\rho$ is modular of weight $\sigma$.
\end{proof}

Level-lowering theorems of Fujiwara~\cite{fujiwara}, Rajaei~\cite{rajaei}
and the third author~\cite{frazer:mazur, frazer:lev} can be viewed as
partial results in the direction of Conjectures~\ref{conj:definite}
and~\ref{conj:indefinite}, of which they are also consequences.

\begin{proposition} \label{prop:conductor}  
Let $\rho:G_K \to \GL_2(\Flbar)$ be continuous, irreducible and
totally odd, let $\gn$ be the (prime to $\ell$) conductor of $\rho$ and let
$\gn' = \gn\prod_{\gp|\ell}\gp^2$.\\
a) Suppose that Conjecture~\ref{conj:definite} holds, or that Conjecture~\ref{conj:indefinite} 
holds and $\rho$ is not badly dihedral. Then
$\rho\sim \ol{\rho}_\pi$ for some cuspidal automorphic 
representation $\pi$ of $\GL_2/K$ of weight $(\vec{2},0)$ and
conductor dividing $\gn'$.\\
b) If $[K:\Q]$ is even, suppose that Conjecture~\ref{conj:definite} holds for $K$;
if $[K:\Q]$ is odd, suppose that Conjecture~\ref{conj:indefinite} 
holds for $K$ and that $\rho$ is not badly dihedral if $\ell=2$ or $3$.
If $K$ is unramified at $\ell$ and $\rho|_{G_{K_\gp}}$ arises from
a finite flat group scheme over $\CO_{K,\gp}$ for each $\gp|\ell$,
then $\rho\sim \ol{\rho}_\pi$ for some cuspidal automorphic 
representation $\pi$ of $\GL_2/K$ of weight $(\vec{2},2)$ and
conductor $\gn$.
\end{proposition}
\begin{proof} To prove (a) assuming Conjecture~\ref{conj:definite}, let $D$ be a quaternion algebra over $K$ ramified
at all infinite places and at most one prime $\gp_0$ over $\ell$, but no other finite
places. Let $U=\prod_\gq U_\gq$ be the open compact subgroup of $D_f^\times$
with
$$U_\gq=\left\{\left.\,\smallmat{a}{b}{c}{d}\in \GL_2(\CO_{K,\gq})\,\right|\,
c\equiv d-1 \equiv 0 \bmod \gn\CO_{K,\gq}\,\right\}$$
for $\gq$ not dividing $\ell$, and $U_\gq$ a pro-$\ell$-Sylow subgroup of
a maximal compact subgroup of $D_\gq^\times$ for $\gq|\ell$.
It follows from Emerton's characterisation of $\pi_\gq$ in
Theorem~\ref{thm:emerton} that $\pi_\gq^{U_\gq} \neq 0$ for all $\gq$ not dividing
$\ell$. The same is true for $\gq|\ell$ since $U_\gq$ is pro-$\ell$ and 
$\pi_\gq\neq 0$. Therefore $S^D(U)[\gm_\rho] = S^D[\gm_\rho]^U \neq 0$
and $\rho \sim \ol{\rho}_{\pi'}$ for some cuspidal automorphic representation
$\pi'$ of $D^\times$ of weight $(\vec{2},0)$ with ${\pi'}^U\neq 0$.
Then the cuspidal automorphic representation $\pi$
of $\GL_2/K$ corresponding to $\pi$ via Jacquet-Langlands has conductor
dividing $\gn'$. (Note that the form of $U_\gq$ for $\gq|\ell$ implies
that $\pi_\gq$ has conductor dividing $\gq^2$.)

The proof of (a) assuming Conjecture~\ref{conj:indefinite} is similar, except
that we use the assumption that $\rho$ is not badly dihedral in order to apply
Lemma~\ref{lemma:elliptic}.

The proof of (b) is also similar, using the fact that if
$\rho|_{G_{K_\gp}}$ arises from a finite flat group scheme, then $W_\gp(\rho)$
contains the trivial representation.  (Recall that we are assuming $\ell$ to be unramified
in $K$, so that badly dihedral representations only occur if $\ell=2$ or $3$, as remarked above.)
\end{proof}

We remark that in fact the conclusions of Proposition~\ref{prop:conductor}
follow from weaker modularity conjectures than Conjectures~\ref{conj:definite}
and~\ref{conj:indefinite}, together with known level lowering results,
at least if $\ell > 2$ and $\rho$ is not badly dihedral. Indeed,
we will explain that weak modularity (Conjecture~\ref{conj:weak})
implies Proposition~\ref{prop:conductor}(a), given these level lowering results. 
It seems that Proposition~\ref{prop:conductor}(b) is
more subtle in that it requires more control over the level at $\ell$, but
these follow from Conjecture~\ref{conj:weak} and
 our weight conjecture (Conjecture~\ref{conj:weight}), together with
level lowering. (In fact, the only part of Conjecture~\ref{conj:weight} we
need is the case where $\rho$ is finite at $\mathfrak{p}$, so that
$W_{\mathfrak{p}}(\rho)$ contains the trivial representation.)

To deduce (a) from the weak conjecture,
we can assume that $\rho\sim\overline{\rho}_\pi$ for some automorphic
representation $\pi$ of $\mathrm{GL}_2/K$ of some weight and level, by
weak modularity. By Corollary~\ref{cor:equiv}, we can assume that the
weight is $(\vec{2},0)$, and that the level is 
$\mathfrak{m}\prod_{\mathfrak{p}|\ell}\mathfrak{p}^{a_{\mathfrak{p}}}$ for some
ideal $\gm$ and some integers $a_{\mathfrak{p}}$. The same argument as in the proof of
Proposition~\ref{prop:conductor}(a) above 
gives that $a_{\mathfrak{p}}\le2$. For primes $\mathfrak{q}\nmid\ell$, we
may use existing level lowering results to deduce that we may take
$\mathfrak{m}=\mathfrak{n}$, so that $\rho$ is modular of weight $(\vec{2},0)$
and level $\mathfrak{n}'$, as required. For $[K:\Q]$ odd, these are due to the
third author and to Rajaei~\cite{frazer:mazur, frazer:lev, rajaei}, under the
technical hypothesis that $\ell > 2$ and $\rho$ is not badly dihedral. When $[K:\Q]$ is
even, a similar argument applies, but in order to use the level lowering results
mentioned above, one needs to begin by raising the level by adding a
prime, using Taylor's theorem~\cite{rlt:inv}. One can then switch to an
appropriate quaternion algebra to perform the level lowering, and finally
remove the prime that we added, using Fujiwara's unpublished 
version~\cite{fujiwara} of Mazur's Principle in the case $[K:\Q]$ even.

Proposition~\ref{prop:conductor}(b) would work in the same way, given
sufficiently strong level lowering statements for primes~$\mathfrak{p}|\ell$.
However, these are not yet sufficient to deduce~(b) from weak
modularity and level lowering. But Conjecture~\ref{conj:weak} and
Conjecture~\ref{conj:weight}, together with level lowering (and the results of
Taylor and Fujiwara when $[K:\Q]$ is even), is sufficient to deduce
Proposition~\ref{prop:conductor}(b); one simply uses the observation made in the
course of the proof above that the trivial representation lies in 
$W_{\mathfrak{p}}(\rho)$ if $\rho$ is finite at $\mathfrak{p}$.

\begin{corollary}\label{cor:finite} If Conjecture~\ref{conj:definite} or~\ref{conj:indefinite} 
holds, then there are only finitely many continuous, irreducible, totally odd
$\rho:G_K \to \GL_2(\ol{\F}_\ell)$ of conductor dividing $\gn$.
\end{corollary}
\begin{proof} Note that by class field theory, there are only finitely
many badly dihedral $\rho$ of a given conductor.  We can therefore
assume $\rho$ is not badly dihedral, so by Proposition~\ref{prop:conductor}(a),
either conjecture implies $\rho$ is modular of weight $(\vec{2},0)$ and level 
$\ell^2.\mathfrak{n(\rho)}$, where $\mathfrak{n}(\rho)$ denotes the Artin
conductor of $\rho$. Since there are only finitely many automorphic
representations of weight $(\vec{2},0)$ and given bounded level, the result
follows.
\end{proof}

\begin{corollary}\label{cor:stw} Suppose that Conjecture~\ref{conj:definite} holds if
$[K:\Q]$ is even, and Conjecture~\ref{conj:indefinite} holds if $[K:\Q]$ is odd.
If $E$ is an elliptic curve over $K$, then $E$ is modular.
\end{corollary} 

\begin{proof}
Given $E$ of conductor $\mathfrak{n}$, let $\ell$ run through all primes greater
than~$3$ and unramified in $K$, such that $E$ has good reduction at all
$\mathfrak{p}|\ell$. Then $\rho_{E,\ell}$ is finite at $\mathfrak{p}$,
so Proposition~\ref{prop:conductor}(b) implies that $\rho_{E,\ell}$ is
modular of weight $(\vec{2},2)$ and level equal to $\mathfrak{n}(\rho_{E,\ell})$, which
divides the conductor of $E$. So there is an automorphic representation
$\pi^{(\ell)}$ of level $U_1(\mathfrak{n})$ and weight $(\vec{2},2)$ whose
mod~$\ell$ representation agrees with $\rho_{E,\ell}$, or equivalently one of
weight $(\vec{2},0)$ giving rise to $\rho_{E,\ell}(-1)$.
There are only finitely many such automorphic representations, so there is a $\pi$
such that $\pi=\pi^{(\ell)}$ for infinitely many $\ell$.  It follows
that for all $\gp$ not dividing $\gn$, the action of $T_\gp$ on
$\pi^{U_1(\gn)}$ is by $a_\gp(E)$ and that of $S_\gp$ is trivial.
Therefore $\rho_{E,\ell}(-1) \sim \rho_\pi$ (for any $\ell$), and
hence $L(E,s) = L(\pi,s)$.
\end{proof}

\begin{remark}  The remarks after Proposition~\ref{prop:conductor} show that
the conclusion of Corollary~\ref{cor:finite} for $\ell > 2$ actually follows
from Conjecture~\ref{conj:weak} and known level lowering results.
Similarly one sees that the conclusion of Corollary~\ref{cor:stw} follows
from Conjecture~\ref{conj:weak}, Conjecture~\ref{conj:weight} and level lowering results.  In fact the modularity
of $E$ even follows from Conjecture~\ref{conj:weak} using modularity lifting results
of Skinner-Wiles~\cite{skinner-wiles:irred}, Fujiwara~\cite{fujiwara:old} or Taylor~\cite{rlt:pm}.
Furthermore, using the lifting results of Kisin~\cite{kisin:annals} and Gee~\cite{gee:mrl}
one can show (unconditionally) that if $\rho_{E,3}$ is irreducible and not badly dihedral,
then $E$ is modular.  See also \cite{skinner-wiles:red}, \cite{skinner-wiles:irred} and
\cite{frazer-jayanta} for additional cases where modularity of $E$
is known.
\end{remark}

\end{document}